\documentclass[11pt]{ip-journal}

\title[]{Infinitesimal Zariski closures of positive representations}

\author[]{Andr\'es Sambarino}
\thanks{The author was partially financed by ANR DynGeo ANR-16-CE40-0025}
\date{}

\usepackage[colorlinks = true,backref = page,hyperindex,breaklinks]{hyperref} 
\usepackage{ stmaryrd }
\usepackage{amssymb}
\usepackage{mathtools}      
\usepackage{mathabx}        
\usepackage[bb = fourier,cal = euler,scr = rsfs]{mathalfa}	
\usepackage{enumitem}       
\usepackage[english]{babel}
\usepackage{tikz}			
\usepackage{tikz-cd}	
\usepackage{tkz-euclide}
\usetikzlibrary{shapes.geometric}
\tikzcdset{arrows=dash}
\usepackage[font = small]{caption}	
\usepackage{nicefrac}
\usetikzlibrary{calc}
\usepackage[percent]{overpic}
\usepackage{mathdots}
\usepackage[mark=o]{dynkin-diagrams}
\usepackage{array,multirow}
\usepackage{graphicx}
\usepackage{afterpage}
\usepackage[section]{placeins}
\usepackage{nicematrix}
\usepackage{ dsfont }
\usepackage{rank-2-roots}

\renewcommand*{\backref}[1]{}
\renewcommand*{\backrefalt}[4]{\quad \tiny
  \ifcase #1 (\textbf{NOT CITED.})%
  \or    (Cited on page~#2.)%
  \else   (Cited on pages~#2.)%
  \fi}

\makeatletter

\def\MRbibitem{\@ifnextchar[\my@lbibitem\my@bibitem}

\def\mybiblabel#1#2{\@biblabel{{\hyperref{http://www.ams.org/mathscinet-getitem?mr=#1}{}{}{#2}}}}

\def\myhyperanchor#1{\Hy@raisedlink{\hyper@anchorstart{cite.#1}\hyper@anchorend}}

\def\my@lbibitem[#1]#2#3#4\par{%
  \item[\mybiblabel{#2}{#1}\myhyperanchor{#3}\hfill]#4%
  \@ifundefined{ifbackrefparscan}{}{\BR@backref{#3}}%
  \if@filesw{\let\protect\noexpand\immediate
    \write\@auxout{\string\bibcite{#3}{#1}}}\fi\ignorespaces%
}

\def\my@bibitem#1#2#3\par{%
  \refstepcounter\@listctr
  \item[\mybiblabel{#1}{\the\value\@listctr}\myhyperanchor{#2}\hfill]#3%
  \@ifundefined{ifbackrefparscan}{}{\BR@backref{#2}}%
  \if@filesw\immediate\write\@auxout
    {\string\bibcite{#2}{\the\value\@listctr}}\fi\ignorespaces%
}

\makeatother

\newcommand{\xqedhere}[2]{%
  \rlap{\hbox to#1{\hfil\llap{\ensuremath{#2}}}}}

\newcommand{\Z}{\mathbb{Z}} 

\newcommand{\R}{\mathbb{R}}
\newcommand{\C}{\mathbb{C}}
\newcommand{\N}{\mathbb{N}}
\renewcommand{\P}{\mathbb{P}}
\newcommand{\T}{\mathbb{T}}

\DeclareMathOperator{\inv}{inv}

\newcommand{\lb}{\llbracket}
\newcommand{\rb}{\rrbracket}
\newcommand{\eps}{\varepsilon}
\newcommand{\G}{\sf{\Gamma}}    

\newcommand{\<}{\langle}
\renewcommand{\>}{\rangle}

\newcommand{\g}{\gamma}

\newcommand{\bord}{\partial}

\newcommand{\posgen}{\cal F^{(2)}}
\newcommand{\posF}{{\cal F}_{>0}}

\newcommand{\vt}{\vartheta}

\renewcommand{\L}{\Lambda}
\newcommand{\p}{{\mathsf{f}}}

\newcommand{\wk}{\check}
\renewcommand{\aa}{\alpha}
\newcommand{\bb}{\beta}

\newcommand{\grupo}{\Lambda} 

\newcommand{\Bcone}{{\cal L}}

\newcommand{\LL}{\mathrm{L}}

\newcommand{\diag}{{\mathrm{diag}}}
\newcommand{\scr}{\mathscr}
\renewcommand{\sf}[1]{{\mathsf{#1}}}
\renewcommand{\rm}{\mathbf}
\newcommand{\cal}{\mathcal}
\renewcommand{\frak}{\mathfrak}

\newcommand{\Lim}[1]{\mathbf{L}_{#1}}

\newcommand{\A}{{\sf A}}
\newcommand{\B}{{\sf B}}
\newcommand{\Ce}{{\sf C}}
\newcommand{\D}{{\sf D}}
\newcommand{\EE}{{\sf E}}
\newcommand{\F}{{\sf F}}
\newcommand{\Ge}{{\sf G}}

\newcommand{\peso}{\varpi}
\newcommand{\om}{\epsilon}
\newcommand{\yy}{\sf{y}}
\newcommand{\xx}{\sf{x}}
\newcommand{\y}{y}
\newcommand{\x}{x}
\newcommand{\hh}{h}
\newcommand{\ee}{e}

\renewcommand{\t}{\vartheta}

\newcommand{\centro}{{\frak Z}}

\DeclareMathOperator{\level}{level}

\newcommand{\fund}{{\phi}}
\newcommand{\Fund}{{\bar{\phi}}}

\newcommand{\mapa}{\Phi}

\newcommand{\Borel}{B}

\newcommand{\upupsilon}{h}

\renewcommand{\b}{{\frak b}}
\newcommand{\n}{{\frak n}}
\newcommand{\gl}{{\frak{gl}}}
\renewcommand{\sl}{\frak{sl}}
\renewcommand{\ge}{{\frak g}}

\renewcommand{\l}{{\frak l}}
\newcommand{\e}{{\frak e}}
\newcommand{\f}{{\frak f}}
\renewcommand{\a}{{\frak a}}
\newcommand{\h}{{\frak h}}

\renewcommand{\p}{{\frak p}}
\renewcommand{\k}{{\frak k}}
\newcommand{\so}{\frak{so}}
\renewcommand{\sp}{\frak{sp}}

\newcommand{\PP}{\hom_{\gtrsim}}

\newcommand{\Weyl}{\cal W}

\newcommand{\simple}{{\sf\Delta}}
\renewcommand{\root}{{\sf\Phi}}
\newcommand{\poids}{{\sf\Pi}}
\newcommand{\sroot}{{\sigma}}

\newcommand{\hasse}[2]{\operatorname{\mathcal{H}}^{#1}_{#2}}
\newcommand{\jordan}{\lambda}
\DeclareMathOperator{\Int}{Int}
\DeclareMathOperator{\isom}{Isom}

\DeclareMathOperator{\ii}{i}
\DeclareMathOperator{\spa}{span}

\DeclareMathOperator{\SL}{{\mathsf{SL}}}
\DeclareMathOperator{\PSL}{{\mathsf{PSL}}}
\DeclareMathOperator{\GL}{{\mathsf{GL}}}
\DeclareMathOperator{\SO}{{\mathsf{SO}}}
\DeclareMathOperator{\PGL}{{\mathsf{PGL}}}

\DeclareMathOperator{\PSO}{{\mathsf{PSO}}}
\DeclareMathOperator{\id}{id}
\DeclareMathOperator{\inte}{int}

\DeclareMathOperator{\lie}{Lie}

\DeclareMathOperator{\ad}{ad}
\DeclareMathOperator{\Ad}{Ad}
\DeclareMathOperator{\Rad}{Rad}

\DeclareMathOperator{\rk}{rank}

\DeclareMathOperator{\Fix}{Fix}
\DeclareMathOperator{\het}{ht}

\newcommand{\auto}{\varkappa}

\newcommand{\circulo}{\partial X_\G}
\newcommand{\gripe}{X}

\newcommand{\base}{{\bf B}}

\newcommand{\mult}{{\sf m}}
\newcommand{\pin}{\mathrm{O}}

\newcommand{\BC}{\vec{\cal E}}

\newcommand{\cartan}{o}

\newcommand{\grass}{\mathrm{Gr}}

\newcommand{\Wedge}{{\bar{\psi}}}  

\DeclareMathOperator{\Sp}{{\sf{Sp}}}

\DeclareMathOperator{\PSp}{{\sf{PSp}}}

\newcommand{\bpm}{\begin{pmatrix}}
\newcommand{\epm}{\end{pmatrix}}

\hypersetup{bookmarksdepth  =  3} 
\numberwithin{equation}{section}     

\setlist[enumerate,1]{label = {\upshape(\roman*)},ref = \roman*}
\setlist[enumerate,2]{label = {\upshape(\alph*)},ref = \alph*}

\newcommand{\cev}[1]{\reflectbox{\ensuremath{\vec{\reflectbox{\ensuremath{#1}}}}}}

\newcommand{\CB}{\cev{\cal E}}

\newtheorem{thmA}{Theorem}

\newtheorem{thm}{Theorem}[section]

\newtheorem{cor}[thm]{Corollary}
\newtheorem{lemma}[thm]{Lemma}
\newtheorem{prop}[thm]{Proposition}

\theoremstyle{definition}

\newtheorem{defi}[thm]{Definition}

\newtheorem{ex}[thm]{Example}

\theoremstyle{remark}
\newtheorem{obs}[thm]{Remark}
\newtheorem*{ack}{Acknowledgements}

\begin{document}

\begin{abstract}We classify the (semi-simple parts of the) Lie algebra of the Zariski closure of a discrete subgroup of a split simple real-algebraic Lie group, whose limit sets are minimal and such that the limit set in the space of full flags contains a positive triple of flags (as in Lusztig \cite{Lusztig-TP}). We then apply our result to obtain a new proof of Guichard's classification \cite{clausura} of Zariski closures of Hitchin representations into $\PSL_d(\R).$
\end{abstract}

\maketitle

\tableofcontents

\section{Introduction}\label{intro}

Let us consider the vector space $\R^d$ equipped with its canonical ordered basis $\cal E=\{e_1,\ldots,e_d\}$ and let $\GL_d(\R)$ be the group of invertible matrices. A \emph{minor} of $g\in\GL_d(\R)$ is the determinant of a square matrix obtained from $g$ by deleting some lines and columns from it. Minors appear naturally when one considers the exterior powers of $\R^d.$ Indeed, theses spaces carry also a natural basis $$\wedge^k\cal E=\{e_{i_1}\wedge\cdots\wedge e_{i_k}:i_1<\cdots<i_k\}$$ defined from $\cal E,$ and the coefficients of $\wedge^kg$ in this basis are the $k\times k$ minors of $g.$

As introduced by Schoenberg \cite{Schoenberg} and Gantmacher-Krein \cite{GK}, a matrix is \emph{totally positive} if all its minors are positive\footnote{Let us convene throughout the paper that $0$ is not a positive real number.}. If $g\in\GL_d(\R)$ is such a matrix, then, since all its entries are positive, it preserves the sharp convex cone of $\R^d$ $$\scr C_{\cal E}=\{(x_1,\ldots,x_d):x_i\geq0\},$$ consisting on vectors all of whose entries in $\cal E$ are non-negative. By the preceding paragraph more is true: the same holds for every exterior power of $g$, $$(\wedge^kg)(v_1\wedge\cdots\wedge v_k)=gv_1\wedge\cdots \wedge gv_k,$$ replacing $\cal E$ by $\wedge^k\cal E.$  

An application of the classical Perron-Frobenius Theorem implies then that $\wedge^kg$ has a unique attracting fixed line in the interior of this cone, \begin{equation}\label{encono}g_{+,k}\in\inte\scr C_{\wedge^k\cal E},\end{equation} and the collection $(g_{+,k})_1^d$ is an attracting complete flag\footnote{Recall that a \emph{complete flag} of $\R^d$ is a sequence of vector subspaces $(V_i)_1^d$ such that $\dim V_i=i$ and $V_i\subset V_{i+1}.$} of $g.$ If we denote by $\BC$ the complete flag $$\BC=(\spa(e_1\oplus\cdots\oplus e_k)\big)_1^d$$ then the inclusion (\ref{encono}) readily implies that the \emph{lower triangular matrix} $\wk u_g$ sending $\BC$ to $(g_{+,k})$ has positive minors (except those that are forced to be zero by the virtue of being lower triangular). Such a semi-group will be denoted by $\wk{U}_{>0}.$ If one is more familiar with \emph{upper triangular matrices} then one should replace $\BC$ by $\CB=(\spa(e_d\oplus\cdots\oplus e_{d-k+1})\big)_1^d$ to obtain an analogous $U_{>0}.$ The subspace of \emph{positive flags} is then defined by $$\cal F_{>0}=\wk{U}_{>0}\cdot\BC= U_{>0}\cdot\CB.$$ The pair of flags $(\BC,\CB)$ uniquely determines the (ordered) decomposition $\R^d=\bigoplus_{e\in\cal E}\R e$, so $\cal F_{>0}$ is actually defined by the pair $(\BC,\CB)$ and \emph{a pinning} (see \S\,\ref{total>0}).


The above (very quick) picture has been generalized to the real points of an arbitrary (Zariski-connected) reductive split real-algebraic group $\sf G$ by Lusztig \cite{Lusztig-TP}. We refer the reader to \S \ref{>0} for the precise definitions and we reuse the notation $\cal F_{\sf G}=\cal F$ as the complete flag space of $\sf G$ and $\cal F_{>0}$ for the subset of positive flags associated to a pair of fixed opposite Borel subgroups $\Borel$ and $\wk{\Borel}$ (and a pinning). Let us say that a triple of pairwise transverse flags $(x,y,z)$ is \emph{positive}, if there exists $g\in\sf G$ such that $g\cdot x=[\wk{\Borel}],$ $g\cdot z=[\Borel]$ and $g\cdot y\in\cal F_{>0}.$

Le us consider more generally a \emph{partial flag} $\cal F_\theta$ of $\sf G,$ these are indexed by subsets of the set of simple roots $\simple,$ with $\cal F_\simple=\cal F.$  An element $g\in\sf G$ is \emph{proximal on} $\cal F_\theta$ if it has an attracting fixed point on $\cal F_\theta,$ i.e.  there exists $g_{+,\theta}\in\cal F_\theta$ fixed by $g$ and an open neighborhood $V$ of $g_{+,\theta}$ such that $g\overline V\subset\inte V.$ In this situation one has $\bigcap_{n\in\N} g^nV=\{g_{+,\theta}\}.$ Elements that are proximal on $\cal F$ are often called \emph{purely loxodromic}.

If $\grupo<\sf G$ is a discrete subgroup then its \emph{limit set on} $\cal F_\theta$ is defined as $$\Lim{\grupo,\theta}=\overline{\{g_{+,\theta}:g\in\grupo\textrm{ proximal on $\cal F_\theta$}\}}\subset\cal F_\theta.$$ A result by Benoist\footnote{(that holds when $\sf G$ is an arbitrary reductive real-algebraic Lie group of non-compact type)} \cite{limite} asserts that if $\grupo$ is Zariski dense, then $\Lim{\grupo,\theta}$ is non-empty and contained in any closed non-empty $\grupo$-invariant set. We will assume a slightly weaker version of this property. Let us say that $\Lim{\grupo,\theta}$ is \emph{minimal} if the only closed $\grupo$-invariant subsets of $\Lim{\grupo,\theta}$ are either the empty set or $\Lim{\grupo,\theta}$ itself. 


\begin{defi}Let $\grupo<\sf G$ be a discrete group. We say that\begin{itemize}\item[-] $\grupo$ has \emph{minimal limit sets} if $\Lim{\grupo,\{\sroot\}}$ is minimal for every $\sroot\in\simple,$ \item[-]$\Lim{\grupo,\simple}$ \emph{contains a positive loxodromic triple} if there exists $g_0\in\grupo$ proximal on $\cal F$ and $x_0\in\Lim{\grupo,\simple}$ such that $(g_+,x_0,g_-)$ is a positive triple.\end{itemize}\end{defi}

Recall that a reductive Lie algebra $\h$ splits as the sum $\h=\h_{ss}\oplus \centro(\h)$ where $\centro(\h)$ is its center and $\h_{ss}=[\h,\h]$ is semi-simple. Recall also that, as $\ge$ is split, it contains a special conjugacy class of sub-algebras isomorphic to $\sl_2(\R)$ called \emph{the principal $\sl_2(\R)$'s}, see \S\ref{xalpha} for the definition.

The main purpose of this paper is to prove the following.

\begin{thmA}\label{tA}Let $\sf G$ be the real points of a Zariski connected, simple split, real-algebraic group and $\grupo<\sf G$ a subgroup with reductive Zariski closure $\sf H,$ minimal limit sets and such that $\Lim{\grupo,\simple}$ contains a positive loxodromic triple. Then $\h_{ss}$ is either $\ge,$ a principal $\sl_2(\R),$ or $\Int\ge$-conjugated to one of the possibilities listed in Table \ref{unesco}.
\end{thmA}

We would like to stress the fact that only one positive (loxodromic) triple in the limit set $\Lim{\grupo,\simple}$ is required.

\begin{table}[h!]
  \begin{center}
    \begin{tabular}{r|r|r} 
      $\ge$ & $\h_{ss}$ & $\phi:\h_{ss}\to\ge$  \\
      \hline
     $\sl_{2n}(\R)$ & $\sp(2n,\R)$ & defining representation \\\hline
      \multirow{2}{*}{$\sl_{2n+1}(\R)$} & $\so(n,n+1)$ $\forall n$ & defining representation \\
       & $\ge_2$ if $n=3$ & fundamental for the short root\\\hline
      $\so(3,4)$ & $\ge_2$ & fundamental for the short root \\\hline
      \multirow{4}{*}{$\so(n,n)$} & $\so(n-1,n)$ $\forall n\geq3$ & stabilizer of a non-isotropic line \\
       & $\so(3,4)$ if $n=4$ & fundamental for the short root \\\cline{3-3}
       & \multirow{2}{*}{$\ge_2$ if $n=4$} & stabilizes a non-isotropic line $L$ and is  \\ & & fundamental for the short root on $L^\perp$\\\hline 
      $\e_6$ & $\f_4$ & $\Fix(\inv_0)$ (see Example \ref{Inv})
    \end{tabular}
    \caption{The statement of Theorem \ref{tA}, if a simple split algebra $\ge$ is not listed in the first column then the only possibilities for $\h_{ss}$ are $\ge$ or a principal $\sl_2(\R).$ The notations $\e_6,\f_4$ and $\ge_2$ refer to the split real forms of the corresponding exceptional complex Lie algebras. 
Observe that there are two non $\Int\so(n,n)$-conjugated embeddings $\so(n,n-1)\to\so(n,n)$ that stabilize a non-isotropic line.} \label{unesco}
  \end{center}
\end{table}

%
%
%
%
The use of Lusztig's positivity to study discrete groups seems to have originated in Fock-Goncharov's \cite{FG} work, where the notion of \emph{positive representation} of a surface group was introduced. A similar approach simultaneously originated in Labourie \cite{labourie}. Both works focus on understanding a special connected component of the character variety $\frak X(\pi_1S,\sf G)=\hom(\pi_1S,\sf G)/\sf G,$ for a closed connected orientable surface $S$ of genus $\geq2$ and a center-free split simple group $\sf G,$ introduced by Hitchin \cite{hitchin}. These \emph{Hitchin components} are defined as those components that contain a discrete and faithful representation $\pi_1S\to\sf G$ such that the Zariski closure of $\rho(\pi_1S)$ is a principal $\PSL_2(\R)$ in $\sf G.$

Combining \cite{FG} and \cite{labourie}, together with Guichard \cite{guichard}, one has the following geometric characterization of Hitchin representations. Recall that the Gromov boundary of $\pi_1S$ is homeomorphic to a circle and carries a $\pi_1S$-invariant cyclic order.

\begin{thm}[{\cite{FG,guichard,labourie}}]\label{HP} A representation $\rho:\pi_1S\to\sf G$ lies in a Hitchin component if and only if there exists a continuous equivariant map $\xi:\partial\pi_1S\to\cal F$ sending cyclically ordered triples to positive triples of flags.\end{thm}

In this paper we deal with a weaker notion than the one required in the above result. We replace $\pi_1S$ with any discrete group acting on a Gromov-hyperbolic space and relax the ``order preserving'' condition. 

If $\gripe$ is a proper Gromov-hyperbolic space and $\G<\isom(\gripe)$ is a discrete subgroup, then we denote by $\circulo$ its limit set on the visual boundary of $\gripe.$ It is a compact $\G$-invariant subset and $\G$ is \emph{non-elementary} if $\circulo$ contains at least 3 points. If this is the case, $\circulo $ is characterized by being the smallest non-empty $\G$-invariant closed subset of $\bord \gripe$, and $\G$ necessarily contains a non-abelian free subgroup. We refer the reader to Ghys-de la Harpe \cite[Chapitre 8]{ghysharpe} for these and other general facts we will require. Unless $\G$ is convex co-compact, the limit set $\circulo$ need not be an intrinsic object associated to the group structure of $\G.$

We will consider the following representations.

\begin{defi}\label{posiDefi}Let $\gripe$ be a proper Gromov-hyperbolic space and $\G$ be a non-elementary discrete isometry group. A representation $\rho:\G\to\sf G$ is \emph{partially positive} if there exists a $\rho$-equivariant continuous map $\xi:\circulo\to\cal F$ such that for every pair $x\neq z$ in $\circulo,$ there exists $y\in \circulo$ such that $\big(\xi(x),\xi(y),\xi(z)\big)$ is a positive triple.\end{defi}

It is implicit in the definition that distinct pairs of $\circulo$ are mapped to transverse flags. 

Recall from Knapp \cite[Chapter B.1]{knapp} that a finite dimensional real Lie algebra $\frak l$ is a semi-direct product $\l_{ss}\oplus_{\pi}\Rad \frak l,$ where $\l_{ss}$ is semi-simple and $\Rad \frak l$ is solvable. The second main result of this paper is the following. 

\begin{thmA}\label{loOtro}Let $\gripe$ be a proper Gromov-hyperbolic space, $\G<\isom X$ a non-elementary discrete subgroup and $\rho:\G\to\sf G$ a partially positive representation. Denote by $\sf L$ the Zariski closure of $\rho(\G).$ Then the semi-simple part $\l_{ss}$ is either $\ge,$ a principal $\sl_2(\R),$ or $\Int\ge$-conjugated to one of the possibilities listed in Table \ref{unesco}.\end{thmA}

The challenge here is to show that $\Lim{\rho(\G),\simple}=\xi(\circulo)$ and that for every $\sroot\in\simple,$ it projects surjectively to every $\Lim{\rho(\G),\{\sroot\}}$ under the natural projection $\cal F\to\cal F_{\{\sroot\}}.$ 

Let us remark that, in contrast with Theorem \ref{tA}, we do not require the Zariski closure of $\rho(\G)$ to be reductive. We emphasize this by stating the following consequence of Theorem \ref{loOtro}, recall that a discrete group acts \emph{strongly irreducibly} on $\R^k$ if it does not preserve a finite collection of non-trivial subspaces.
 
\begin{cor}\label{reductiveClosure} Assume that $\ge=\sl_n(\R),\sp_{2n}(\R),$ $\so(n,n+1)$ or $\ge_2$. Let $\gripe$ be a proper Gromov-hyperbolic space, $\G<\isom X$ a non-elementary discrete subgroup and $\rho:\G\to\sf G$ a partially positive representation, then its corresponding action on $\R^n,\R^{2n},\R^{2n+1}$ or $\R^7$ respectively is (strongly) irreducible.\end{cor}


Theorem \ref{loOtro} together with Theorem \ref{HP}  give a new proof of the following classification result by Guichard (the argument is postponed to \S\ref{gr}). As before, $\ge_2$ is the split real form of the corresponding complex exceptional Lie algebra and $\Ge_2=\Int \ge_2.$ 

\begin{cor}[Guichard \cite{clausura}]\label{gui} Let $\rho:\pi_1S\to\PSL_d(\R)$ be a representation in the Hitchin component. Then $\rho(\pi_1S)$ is contained in the identity component of its Zariski closure, and these are: either $\PSL_d(\R),$ a principal $\PSL_2(\R)$ or conjugated to one of the following:\begin{itemize}\item[-]$\PSp_{2n}(\R)$ if $d=2n$ for all $n\geq1,$\item[-] $\PSO_0(n,n+1)$ if $d=2n+1$ for all $n\geq1,$ \item[-] the fundamental representation for the short root of $\Ge_2$ if $d=7.$\end{itemize}\end{cor}

Corollary \ref{gui} plays a central role in Corollary 11.8 of Bridgeman-Canary-Labourie-S. \cite{pressure} and in the recent work by Danciger-Zhang \cite{Danciger-Zhang}, allowing the authors to reduce the general problem to the group $\PSO(n,n+1).$

\subsection{Final remarks}

It is unclear whether all possibilities stated in Theorem \ref{loOtro} might actually occur. When $\G=\pi_1S$ ($S$ as above) then Hitchin's Theorem \cite{hitchin} implies this is actually the case. However, a recent result by Alessandrini-Lee-Schaffhauser \cite{HitchinOrbifold} provides many examples of locally rigid positive representations of groups with torsion.

\subsection{Organization of the paper} In \S\ref{remainders} we recall some facts on representation theory of real reductive Lie algebras of non-compact type. In \S\ref{HasseDiagrams} we introduce the Hasse diagram of a representation of such a Lie algebra, this is nothing but the usual Hasse diagram of a partially order set (here to be the restricted weights of the representation with their natural partial order). We introduce maps between diagrams and notably study the existence of a surjective map between two Hasse diagrams. There is a case by case proof that is postponed to appendix \S\ref{figurasdiagramas}. 

In \S\ref{discretegroups} we study Zariski closures of discrete groups verifying a coherence condition with respect to the position of their eigenspaces, and relate these to maps between Hasse diagrams of the Zariski closure and the ambient group. The key point is Proposition \ref{existenMapas} that, in light of the previous section, classifies Zariski closures of these groups, provided it is reductive.

Section \ref{total>0} begins by recalling total positivity introduced by Lusztig \cite{Lusztig-TP}, we prove then that groups whose limit sets contains a positive loxodromic triple verify the coherence condition studied in \S\ref{discretegroups}. This proves Theorem \ref{tA}. Theorem \ref{loOtro} is also proved in this section. 
In \S\ref{gr} we focus on the $\SL_d(\R)$ situation and prove Guichard's classification (Corollary \ref{gui}). 

The paper is written rather linearly so one has the following diagram representing dependence between sections:
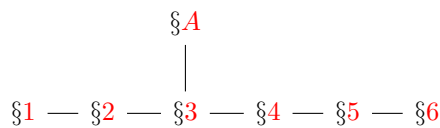
\begin{center}
\begin{tikzcd}[column sep=small,]
&&   \S\ref{figurasdiagramas} \arrow[d]& & &&\\ 
\S\ref{intro}\arrow[r]&\S\ref{remainders}\arrow[r]&\S\ref{HasseDiagrams}\arrow[r]&\S\ref{discretegroups}\arrow[r]&\S\ref{total>0}\arrow[r]&\S\ref{gr}
\end{tikzcd}
\end{center}

\begin{ack}The author would like to thank Olivier Gui\-chard and Maria Beatrice Pozzetti for enlightening discussions and careful reading of this article. He would also like to thank the referees for careful reading and improving the exposition of the paper.
\end{ack}

\section{Review on Lie Theory}\label{remainders}

\subsection{Semi-simple Lie algebras} Let $\ge$ be a semi-simple real Lie algebra of the non-compact type and fix a Cartan involution $\cartan:\ge\to\ge$ with associated Cartan decomposition $\ge=\k\oplus\p.$ Let $\a\subset\p$ be a maximal abelian subspace and let $\root\subset\a^*$ be the set of restricted roots of $\a$ in $\ge.$ For $\aa\in\root$ let us denote by $$\ge_\aa=\{u\in\ge:[a,u]=\aa(a)u\ \forall a\in\a\}$$ its associated root space. One has the (restricted) root space decomposition $\ge=\ge_0\oplus\bigoplus_{\aa\in\root}\ge_\aa,$ where $\ge_0$ is the centralizer of $\a.$

Fix a Weyl chamber $\a^+$ of $\a$ and let $\root^+$ and $\simple$ be, respectively, the associated sets of positive roots and of simple roots. One has that $\root=\root^+\cup-\root^+$ and that if $\aa\in\root^+$ then, upon writing $$\aa=\sum_{\sroot\in\simple}k_\sroot\sroot,$$ every coefficient $k_\sroot$ is a non-negative integer. The \emph{height} of $\aa$ is $\het(\aa)=\sum_\sroot k_\sroot.$

Let us denote by $(\cdot,\cdot)$ the Killing form of $\ge,$ its restriction to $\a,$ and its associated dual form in the dual $\a^*$ of $\a.$ For $\chi,\psi\in\a^*$ define \begin{equation}\label{prodsimply}\<\chi,\psi\>=2\frac{(\chi,\psi)}{(\psi,\psi)}.\end{equation}

The \emph{Weyl group} of $\root,$ denoted by $\Weyl,$ is the group generated by, for each $\aa\in\root,$ the reflection $r_\aa:\a^*\to\a^*$ on the hyperplane $\aa^\perp,$ $$r_\aa(\chi)=\chi-\<\chi,\aa\>\aa.$$ It is a finite group with a unique \emph{longest} element $w_0$ (w.r.t. the word metric on the generating set $\{r_\aa:\aa\in\simple\}$). This longest element sends $\a^+$ to $-\a^+.$

Recall that the \emph{Dynkin diagram} of the root system $\root$ consists on a graph whose vertices are the elements of $\simple$ and such that $\aa,\bb\in\simple$ are joined by $\<\aa,\bb\>\<\bb,\aa\>$ edges. If two simple roots are joined by more than one edge then an arrow is added pointing to the shortest (in norm $(\cdot,\cdot)$) root. One speaks indistinctively of the Dynkin diagram of $\ge,$ $\root$ or of $\simple.$

We will require the following notion:

\begin{defi}\label{extremal}An element of $\simple$ is \emph{extremal} if it is connected to exactly one root in the Dynkin diagram of $\root.$\end{defi}

The root systems of type $\D$ and $\EE$ have 3 extremal roots, while the others only have two.

\subsubsection{Some $\sl_2$'s of $\ge$}\label{xalpha}

For $\aa\in\root$ let $t_\aa,h_\aa\in\a$ be defined such that, for all $v\in\a$  and all $\varphi\in\a^*,$ one has $$\aa(v)=(v,t_\aa)\textrm{ and }\varphi(h_\aa)=\<\varphi,\aa\>.$$ These two vectors are related by the simple formula $h_\aa=2t_\aa/(t_\aa,t_\aa).$ Recall that for $x\in\ge_\aa$ one has $[x,o(x)]=(x,o(x))t_\aa$. Thus, for each $\aa\in\root^+$ and $\xx_\aa\in\ge_\aa$ there exists $\yy_\aa\in\ge_{-\aa}$ such that \begin{alignat*}{3} e=& (\begin{smallmatrix}0 & 1 \\ 0 & 0\end{smallmatrix})& \mapsto \xx_\alpha \\ f= & (\begin{smallmatrix}0 &  0\\ 1 & 0\end{smallmatrix}) & \mapsto \yy_\alpha\\  h= & (\begin{smallmatrix}1 & 0 \\ 0 & -1\end{smallmatrix})& \mapsto h_\alpha \end{alignat*} is a Lie algebra isomorphism between $\sl_2(\R)$ and the span of $\{\xx_\aa,\yy_\aa,h_\aa\}.$ Let us fix such a choice of $\xx_\aa$ and $\yy_\aa$ from now on.

One says that $\ge$ is \emph{split} if the complexification $\a\otimes\C$ is a Cartan subalgebra of $\ge\otimes\C.$ Equivalently, $\ge$ is split if the centralizer $\centro_\k(\a)$ of $\a$ in $\k$ is trivial.

Assume that $\ge$ is split. Following Kostant \cite[\S 5]{kostant}, consider the dual basis of $\{t_\sroot:\sroot\in\simple\}$ relative to $(\cdot,\cdot)$: $(\om_\aa,t_\bb)=\delta_{\aa\bb},$ and let $\om_0=\sum_{\sroot\in\simple}\om_\sroot\in\a.$ The element $\om_0$ is the semi-simple element of a $3$-dimensional simple subalgebra of $\ge.$ Such a subalgebra, or any of its $\Int \ge$-conjugates, will be called \emph{a principal $\sl_2(\R)$} of $\ge.$ 

Let us denote by $\frak n=\bigoplus_{\aa\in\root^+}\ge_\aa.$

\begin{thm}[{Kostant \cite[Thm 5.3]{kostant}}]\label{nilpoPrinc}Let $\ge$ be a split Lie algebra and consider an element $$\ee=\sum_{\aa\in\root^+}a_\aa\xx_\aa\in\frak n.$$ Then $\ee$ lies in a principal $\sl_2(\R)$ if and only if $a_\sroot\neq0$ for all $\sroot\in\simple.$
\end{thm}

\subsection{Reductive groups}

A Lie algebra $\ge$ is \emph{reductive} if every $\ad\ge$-invariant subspace of $\ge$ has an $\ad\ge$-invariant complement. It is a standard fact (see Knapp \cite[Chapter I. \S 7]{knapp}) that such an algebra splits as $$\ge=\centro(\ge)\oplus \ge_{ss},$$ where $\ge_{ss}=[\ge,\ge]$ is semi-simple and $\centro(\ge)$ is the center of $\ge.$

A \emph{reductive Lie group} (see for example Knapp \cite[Chapter VII. \S 2.]{knapp}) $\sf G$ is a 4-tuple $\big(\sf G,\sf K,\sigma,(\cdot,\cdot)\big),$ where $\sf K$ is a compact subgroup of $\sf G,$ $\sigma$ is a Lie algebra involution of $\ge$ and $(\cdot,\cdot)$ is a $\sigma$-invariant, $\Ad \sf G$-invariant non-degenerate bilinear form on $\ge$ such that:

\begin{itemize}
\item[-]$\ge$ is a reductive Lie algebra,
\item[-]the Lie algebra $\k$ of $\sf K$ is the set of fixed points of $\sigma,$
\item[-] if $\p=\{x\in\ge:\sigma(x)=-x\}$ then $\k$ and $\p$ are $(\cdot,\cdot)$-orthogonal and $(\cdot,\cdot)$ is positive definite on $\p,$
\item[-] the map $\sf K\times\p\to\sf G,$ $(k,x)\mapsto k\exp x,$ is a surjective diffeomorphism.
\item[-] every automorphism of the form $\Ad(h),$ for $h\in\sf G,$ of the complexification $\ge\otimes\C$ is of the form $\Ad(x)$ for some $x\in\Int(\ge\otimes\C).$
\end{itemize}

Given a reductive group $\sf G$ and a maximal abelian subspace $\a\subset\p,$ one can form, as in the semi-simple case, a restricted root space decomposition $$\ge=\ge_0\oplus\bigoplus_{\aa\in\root_\ge}\ge_\aa$$ where $\ge_\aa=\{x\in\ge:[a,x]=\aa(a)x\,\forall a\in\a\}.$

The relation between the restricted roots $\root_\ge$ and the restricted roots of $\ge_{ss}$ is as follows: the elements of $\root_\ge$ can be obtained by considering the restricted root space decomposition of $\ge_{ss}$ relative to $\a_{ss}=\a\cap\ge_{ss}$ and extending these roots to $\a$ as being zero on $\a\cap\centro(\ge).$



\subsection{Basic facts on representation theory of semi-simple Lie algebras}\label{reps}

Let $\ge$ be a semi-simple Lie algebra over $\R$ without compact factors. We record here some standard facts about irreducible real representations of $\ge,$ see for example Humphreys \cite{james}.

The \emph{restricted weight lattice} is defined by $$\poids=\{\varphi\in\a^*:\<\varphi,\aa\>\in\Z\ \forall\aa\in\root\},$$ it is spanned by the \emph{fundamental weights}: $\{\peso_\sroot:\sroot\in\simple\}$ where $\peso_\sroot$ is defined by $$\<\peso_\sroot,\bb\>=d_\sroot\delta_{\sroot\bb}$$ for every $\sroot,\bb\in\simple,$ where $d_\sroot=1$ if $2\sroot\notin\root^+$ and $d_\sroot=2$ otherwise. 
The set $\poids_+$ of \emph{dominant restricted weights} is defined by $\poids_+=\poids\cap (\a^+)^*.$

Given $\chi,\psi\in\poids$ one says that $\chi\succ\psi$ if $\chi-\psi$ has non-negative integer coefficients in $\simple.$ A subset $\pi\subset\poids$ is \emph{saturated} if for every $\chi\in\pi$ and $\aa\in\root$ the \emph{string} $$\chi-i\aa\qquad i\textrm{ between $0$ and }\<\chi,\aa\>$$ is entirely contained in $\pi.$ Such a set is necessarily $\Weyl$-invariant. We say that $\pi$ has \emph{highest weight} $\mu\in\pi$ if for every $\chi\in\pi$ one has $\mu\succ\chi.$ One has the following lemma, see Humphreys \cite[\S 13.4 Lemma B]{james}.

\begin{lemma}\label{satu} Let $\pi$ be a saturated set of weights with highest weight $\mu,$ then every $\chi\in\poids_+$ with $\mu\succ\chi$ belongs to $\pi.$
\end{lemma}

Let $\phi:\ge\to\sl(V)$ be an irreducible representation. The sub-algebra $\phi(\a)$ is self-adjoint for an inner product of $V$ and thus the space $V$ decomposes as a sum $V=\bigoplus_{\chi\in\poids(\phi)} V^\chi,$ where $$V^\chi=\{v\in V:\phi(a)v=\chi(a)v\,\forall a\in\a\}$$ are the common eigen-spaces, called \emph{restricted weight spaces}, and $$\poids(\phi)=\big\{\chi\in\a^*:V^\chi\neq\{0\}\big\}$$ is called the \emph{set of restricted weights} of $\phi.$ It is a $\Weyl$-invariant set. The \emph{multiplicity} of $\chi\in\poids(\phi)$ is denoted by $\mult_\phi(\chi)$ and defined as the dimension of its restricted weight space, $\mult_\phi(\chi)=\dim V^\chi.$ We will often omit the subscript and write $\mult(\chi)$ if there no ambiguity in $\phi.$

\begin{prop}[{See Humphreys \cite[Proposition 21.3]{james}}]\label{cuerda}Let $(V,\phi)$ be an irreducible representation of $\ge.$ Then the set $\poids(\phi)$ is saturated with highest weight $\chi_\phi$. In particular, for $\chi\in\poids(\phi)$ and $\aa\in\root,$ the elements of $\poids(\phi)$ of the form $\chi+i\aa,\, i\in\Z$ form an unbroken string $$\chi+i\aa,\, i\in\lb-r,q\rb$$ and $r-q=\<\chi,\aa\>.$\end{prop}

The  unique maximal element $\chi_\phi$ of $\poids(\phi)$  from  Proposition \ref{cuerda} is called the \emph{the highest restricted weight} of $\phi$. By definition, for every $a\in\a^+$ one has $\chi_\phi(a)=\jordan_1\big(\phi(a)\big)$, the spectral radius of $\phi(a)$. The restricted weight space associated to $\chi_\phi$ is \begin{equation}\label{V+}V^+=V^{\chi_\phi}=\big\{v\in V:\phi(\n)v=\{0\}\big\}.\end{equation} 

To simplify notation, for $\aa\in\root^+$, we let $\wk{\ge}_\aa=\ge_{-\aa}$, then one has the following.

\begin{obs}\label{generador}The subspaces of the form $\phi(\wk\ge_{\bb_\ell})\cdots\phi(\wk\ge_{\bb_0})V^+$ with $\bb_i\in\simple$ (repetitions allowed) that do not identically vanish are in direct sum. Indeed, such a space is contained the restricted weight space associated to $$\chi_\phi-\sum_{i=0}^\ell\bb_i.$$ Every weight of $\phi$ is obtained in this fashion, moreover, by construction every weight $\chi\in\poids(\phi)$ can be written as $\chi=\chi_\phi-\bb_0-\cdots-\bb_\ell,$ with $\bb_j\in\simple,$ in such a way that all the partial sums $$\chi=\chi_\phi-\bb_0-\cdots-\bb_j\ \  j\in\lb1,\ell\rb$$ are weights of $\phi.$\end{obs}

\begin{ex} Let us treat the example of the defining representation $\phi$ of $\sl_d(\R)$, i.e. the identity representation $\phi:\sl_d(\R)\to\sl_d(\R)$. A Cartan subspace is  $$\a=\{\diag(a_1,\cdots,a_d): a_i\in\R\textrm{ and }\sum a_i=0\}.$$ A set of simple roots is $\simple=\{\bb_i\}_1^{d-1}$, where for each $i\in\lb1,d-1\rb$ one lets $\bb_i(a)=a_i-a_{i+1}$, and the associated Weyl chamber is $\a^+=\{a\in\a:a_i\geq a_{i+1},\ i\in\lb1,d-1\rb\}$.

The highest weight of the representation $\phi$ is $\chi_\phi\in\a^*$ such that for all $a\in\a^+$ one has $\chi_\phi(a)$ is the spectral radius of $\phi(a)=a$. By the choice of $\a^+$ one has $\chi_\phi=\peso_{\bb_1}:a\mapsto a_1.$ The remaining weights of $\phi$, i.e. the elements of $\a^*$ describing the eigenvalues of $a\in\a$, are $\poids(\phi)=\{\eps_i(a)=a_i\}_1^{d-1}$. We find them algorithmically from $\peso_{\bb_1}$ and $\simple$ by means of Proposition \ref{cuerda} and Remark \ref{generador} as follows:

\begin{itemize}\item[1)] Consider the simple roots $\sroot$ that are not orthogonal to $\chi_\phi$ (equivalently such that $\<\chi_\phi,\sroot\>\neq0$). In this case only $\sroot=\bb_1$ works, giving $\<\chi_\phi,\bb_1\>=1$ by the very definition of $\chi_\phi=\peso_{\bb_1}$, so the $\bb_1$-string through $\chi_\phi$ has length $1$ which yields that $\chi_2=\peso_\chi-\bb_1:a\mapsto a_1-(a_1-a_2)=a_2$ is a weight of $\phi$.

\item[2)] We now consider the roots $\sroot$ with $\<\chi_2,\sroot\>\neq0$. By linearity of $\<\,,\>$ on the first coordinate one sees that only $\bb_1$ and $\bb_2$ work in this case, with values $-1$ and $1$ respectively. The first one gives that $\chi_2+\bb_1$ is a weight (which we already new), and the second one gives $\chi_3=(\peso_{\bb_1}-\bb_1)-\bb_2:a\mapsto a_3$ is a weight of $\phi$.

\item[3)] one repeats the procedure to obtain the other weights.
\end{itemize}
\end{ex}

\section{Hasse diagrams for representations}\label{HasseDiagrams} If $\phi:\ge\to\frak{gl}(V)$ is an irreducible representation of a real semi-simple Lie algebra $\ge$ without compact factors, then its set of weights carries the partial order $\succ$ previously defined: $\chi\succ\psi$ if the coefficients of $\chi-\psi$ in $\simple$ are non-negative integers.

One defines then the \emph{Hasse diagram} of the representation $\phi$ as a graph whose vertices are the elements of $\poids(\phi),$ and one draws an edge between $\chi$ and $\psi$ if and only if $\chi-\psi\in\simple.$ Because of the non-symmetry of $\succ,$ the edge should be a directed arrow, however we prefer to forget the arrow and draw $\psi$ \emph{below} $\chi.$ It is also convenient to label the edge with the simple root $\chi-\psi.$

These Hasse diagrams carry a natural grading or \emph{levels} defined by the function $$\level\Big(\chi_\phi-\sum_{\sroot\in\simple}k_\sroot\sroot\Big)=1+\sum k_\sroot.$$ By means of Remark \ref{generador} one can draw the Hasse diagram of a given representation \emph{level by level}, starting from it's highest weight and inductively checking, for a given weight $\chi\in\poids(\phi)$ the set of simple roots $\sroot\in\simple$ such that $\phi(\wk{\ge}_\sroot)V^\chi=\{0\}.$ This in turn can be directly computed from the root system $\root$ using Proposition \ref{cuerda}: one computes $\<\chi,\sroot\>$ and, since all lower levels of the diagram are assumed to be known, one knows whether $\chi+\sroot$ (down one level) belongs to $\poids(\phi)$ or not.

It is more convenient then to define the Hasse diagram as depending only on the type of the root system $\root,$ and of a given dominant weight $\chi\in\poids_+$ that will play the role of the highest weight of an irreducible representation. 

\begin{defi}The Hasse diagram of a root system of type $\sf L$ and a given dominant weight $\chi\in\poids_+$ will be denoted by $\hasse{\sf L}{\chi}.$
\end{defi}

\begin{ex}\label{48}For example, the Hasse diagram of a fundamental weight $\peso_\sroot$, where $\sroot$ is such that $2\sroot\notin\simple$, has \begin{itemize}\item[-]solely $\peso_\sroot$ at the first level, \item[-] only $\peso_\sroot-\sroot$ at the second level, \item[-] the forms $\peso_\sroot-\sroot-\bb$, for every $\bb\in\simple$ neighboring $\sroot$ in the Dynkin diagram of the given root system, at the third level.
\end{itemize}
The remaining levels can become quickly very complicated. 
\end{ex}

Figure (\ref{hasseEx}) depicts the Hasse diagrams of the exceptional root system $\Ge_2$ for both its fundamental weights, the Dynkin diagram is added to the picture together with the corresponding set of weights  in each case.

\begin{figure}
\begin{tikzpicture}[baseline= (a).base]
\node[scale=1]  (a) at (-2,0.5){
\begin{tikzcd}[column sep=small,]
 & \circ \arrow[d, "\bb"]&  \\ 
 & \circ \arrow[d, dash, "\aa"]&  \\ 
 & \circ \arrow[d, dash, "\aa"]&   \\
 & \circ \arrow[dl, dash, swap, "\bb"] \arrow[dr, dash,"\aa"]  & \\
  \circ  \arrow[dr, dash, swap,  "\aa"] &  & \circ \arrow[dl, dash, "\bb"]   \\
  & \circ \arrow[dr,dash, "\aa"] \arrow[dl,dash, swap, "\bb"] &   \\
 \circ \arrow[dr, dash, swap, "\aa"] & & \circ \arrow[dl, dash,"\bb"]   \\
  & \circ \arrow[d,dash,"\aa"]  &       \\
  & \circ \arrow[d,dash,"\aa"]  &       \\
  & \circ \arrow[d,dash,"\bb"]  &       \\
  & \circ&
\end{tikzcd}};

\node[scale=1] at (-2,-6){$\hasse{\Ge_2}{\peso_\bb}$};
\node[scale=1] at (2,0){
\begin{tikzcd}[column sep=small]
 & \circ \arrow[d, "\aa"]& & \\ 
  & \circ \arrow[d, "\bb"]& & \\
 & \circ\arrow[d, "\aa"]&  & \\
 & \circ \arrow[d, "\aa"]& & \\ 
  & \circ \arrow[d, "\bb"]& & \\
 & \circ\arrow[d, "\aa"]&  & \\
 & \circ  & &
\end{tikzcd}};\node[scale=1] at (2,-6){$\hasse{\Ge_2}{\peso_\aa}$};

\node[scale=1] at (0,-6.5){\dynkin[labels={\bb,\aa},scale=1.4] G2};

\node[scale=1] (b) at (4,0) {\begin{tikzpicture}\begin{rootSystem}{G}
\node[right] at (hex cs:x=1,y=0){\small\(\peso_\aa\)};
\roots

\wt [black]{1}{0}
\wt [black]{-1}{1}
\wt [black]{-1}{0}
\wt [black]{0}{0}
\wt [black]{1}{-1}
\wt [black]{-1}{1}
\wt [black]{2}{-1}
\wt [black]{-2}{1}
\WeylChamber
\end{rootSystem}\end{tikzpicture}};

\node at (4,-6) {$\poids_{\Ge_2}(\peso_\aa)$};

\node[scale=1] at (-5,0) {\begin{tikzpicture}\begin{rootSystem}{G}

\node[right] at (hex cs:x=1,y=1){\small\(\peso_\bb\)};
\roots
\wt [black]{0}{0}
\wt [black]{1}{0}
\wt [black]{-1}{1}
\wt [black]{-1}{0}
\wt [multiplicity=2,black]{0}{0}
\wt [black]{1}{-1}
\wt [black]{-1}{1}
\wt [black]{2}{-1}
\wt [black]{-2}{1}
\wt [black]{0}{1}
\wt [black]{0}{-1}
\wt [black]{3}{-1}
\wt [black]{-3}{2}
\wt [black]{-3}{1}
\wt [black]{3}{-2}

\WeylChamber
\end{rootSystem}\end{tikzpicture}};

\node at (-5,-6) {$\poids_{\Ge_2}(\peso_\bb)$};

\end{tikzpicture}\caption{Hasse diagrams for fundamental weights of (extremal) roots of $\Ge_2$, together with the corresponding weight sets (in black).}\label{hasseEx}
\end{figure}

\subsection{Maps between diagrams}\label{maps}

Given two root systems of types $\sf J$ and $\sf L,$ consider a function $f:\simple_{\sf L}\to\simple_{\sf J}.$ We will define a \emph{diagram map with labeling $f$}, in short a \emph{diagram map}, between two Hasse diagrams as a function $\T^f:\hasse{\sf L}\chi\to\hasse{\sf J}{\chi'}$ such that if $\psi_0,\psi_1\in\hasse{\sf L}{\chi}$ then $$\psi_0-\psi_1\in\simple_{\sf L}\textrm{ implies }\T^f(\psi_0)-\T^f(\psi_1)=f(\psi_0-\psi_1)\in\simple_{\sf J}.$$

Such a map is thus order preserving, level and labeling equivariant. We say that $\T^f$ is \emph{surjective} if it is set-wise surjective. If this is the case, then necessarily $f$ is surjective and both diagrams have the same total number of levels.

Let us emphasize that the function $f$ is merely a set-wise function, no condition on the associated function between the Dynkin diagrams is required.

\begin{ex}\label{Inv}Consider the following Dynkin diagrams that carry a non-trivial involution, $\inv_0:\simple_{\sf L}\to\simple_{\sf L}$ say,  

\begin{itemize}\item[-] the middle point symmetry in $\A_\ell$: $\begin{dynkinDiagram}A{}\draw[thick] (root 1) to [out=-45, in=-135] (root 4);\draw[thick] (root 2) to [out=-45, in=-135] (root 3);\end{dynkinDiagram},$

\item[-]$\D_n:$ $\begin{dynkinDiagram}D{}\draw[thick] (root 5) to [out=-45, in=45] (root 6);\end{dynkinDiagram},$ 

\item[-] the middle axis symmetry in $\EE_6$: $\begin{dynkinDiagram}E6\draw[thick] (root 1) to [out=-45, in=-135] (root 6);\draw[thick] (root 3) to [out=-45, in=-135] (root 5);\end{dynkinDiagram}.$
\end{itemize}

The quotient by the orbits of $\inv_0$ provides a labeling\begin{itemize}\item[-] $f:\simple_{\A_{2n+1}}\to\simple_{\B_n},$\item[-]$f:\simple_{\D_{n}}\to\simple_{\Ce_n},$\item[-]$f:\simple_{\EE_6}\to\simple_{\F_4},$\end{itemize} which induces surjective maps between the Hasse diagrams of the fundamental weight $\peso_{\sroot}$ of a given simple root and the fundamental weight of $f(\sroot).$ Figure (\ref{E6-F4}) in the appendix depicts the $\EE_6$ case for one of the extremal roots.
\end{ex}




Not every example comes from the fixed point set of an involution, as the fundamental representation $\Fund_{\peso_{\aa}}:\ge_2\to\sl_7(\R)$ of the real split Lie algebra $\ge_2$ shows. This is depicted in Figure (\ref{knappsl}).

\begin{figure}
\begin{tikzpicture}
\node[scale=1] (a) at (-1.5,0){\begin{tikzcd}[column sep=small]
 & \circ \arrow[d, "\bb_1"]& & \\ 
  & \circ \arrow[d, "\bb_2"]& & \\
 & \circ\arrow[d, "\bb_3"]&  & \\
 & \circ \arrow[d, "\bb_4"]& & \\ 
  & \circ \arrow[d, "\bb_5"]& & \\
 & \circ\arrow[d, "\bb_6"]&  & \\
 & \circ  & &
\end{tikzcd}};\node[scale=1] at (-1.5,-4){$\hasse{\A_{6}}{\peso_{\bb_1}}$};

\node[scale=1] at (-1.5,-4.7){\dynkin[labels={\bb_1,,,\bb_6},scale=1.4] A{}};

\node[scale=1] (b) at (2,0){\begin{tikzcd}[column sep=small]
 & \circ \arrow[d, "\aa"]& & \\ 
  & \circ \arrow[d, "\bb"]& & \\
 & \circ\arrow[d, "\aa"]&  & \\
 & \circ \arrow[d, "\aa"]& & \\ 
  & \circ \arrow[d, "\bb"]& & \\
 & \circ\arrow[d, "\aa"]&  & \\
 & \circ  & &
\end{tikzcd}
};
\node[scale=1] at (2,-4){$\hasse{\Ge_2}{\peso_\aa}$};

\node[scale=1] at (2,-4.7){\dynkin[labels={\bb,\aa},scale=1.4] G2};

\draw[scale=1,->] (a) -- (b) node[midway,above] {$\T^f$};

\end{tikzpicture}\caption{The surjective map $\hasse{\A_{6}}{\peso_{\bb_1}}\to\hasse{\Ge_2}{\peso_\aa}.$}\label{knappsl}
\end{figure}

The existence of a surjective map between Hasse diagrams is of course very restrictive as the following lemma shows.

\begin{lemma}\label{diagramasClasif}Consider two irreducible reduced root systems of types $\sf J$ and $\sf L.$ Assume there exists \begin{itemize}\item[-]$f:\simple_{\sf L}\to\simple_{\sf J}$ such that $f(\aa)$ is extremal for every extremal $\aa\in\simple_{\sf L},$ \item[-] for every extremal $\aa$ a surjective diagram map $\T^f:\hasse{\sf L}{\peso_\aa}\to\hasse{\sf J}{\peso_{f(\aa)}}$ with labeling $f.$ \end{itemize} Then, besides $f=$identity, the only possibilities for $\sf J,$ $\sf L,$ and $f$ are listed in Table \ref{duquesa}.
\end{lemma}

\begin{table}[h!]
  \begin{center}
    \begin{tabular}{r|r|r} 
      $\sf L$ & $\sf J$ & fibers of $f$  \\
      \hline
      \multirow{2}{*}{$\A_{2n}$} & $\B_n$ $\forall n$ & $\begin{dynkinDiagram}A{}\draw[thick] (root 1) to [out=-45, in=-135] (root 4);\draw[thick] (root 2) to [out=-45, in=-135] (root 3);\end{dynkinDiagram}$ \\
       & $\Ge_2$ if $n=3$ & Figure (\ref{knappsl})\\\hline
     $\A_{2n-1}$ & $\Ce_{2n}$ & $\begin{dynkinDiagram}A{}\draw[thick] (root 1) to [out=-45, in=-135] (root 4);\draw[thick] (root 2) to [out=-45, in=-135] (root 3);\end{dynkinDiagram}$ \\\hline
      $\B_3$ & $\Ge_2$ & $\begin{dynkinDiagram}B3\draw[thick] (root 1) to [out=-45, in=-135] (root 3);\end{dynkinDiagram}$ \\\hline
      \multirow{4}{*}{$\D_n$} & $\B_{n-1}$ $\forall n\geq3$ & $\begin{dynkinDiagram}D{}\draw[thick] (root 5) to [out=-45, in=45] (root 6);\end{dynkinDiagram}$ \\
       & $\B_3$ if $n=4$ & $\begin{dynkinDiagram}D{4}\draw[thick] (root 1) to [out=-60, in=180] (root 4);\end{dynkinDiagram}$\\\cline{3-3}
       & \multirow{2}{*}{$\Ge_2$ if $n=4$} & $\begin{dynkinDiagram}D{4}\draw[thick] (root 1) to [out=-60, in=180] (root 4);\draw[thick] (root 3) to [out=-60, in=60] (root 4);\draw[thick] (root 1) to [out=60, in=180] (root 3);\end{dynkinDiagram}$\\\hline 
      $\EE_6$ & $\F_4$ & $\begin{dynkinDiagram}E6\draw[thick] (root 1) to [out=-45, in=-135] (root 6);\draw[thick] (root 3) to [out=-45, in=-135] (root 5);\end{dynkinDiagram}$
    \end{tabular}
    \caption{} \label{duquesa}
  \end{center}
\end{table}

\begin{proof} The proof is a case by case verification. In Appendix \ref{figurasdiagramas} we draw the Hasse diagrams for the fundamental weights of the extremal roots of all irreducible reduced root systems and the non-existence verification is also proven.
\end{proof}

To end this section we remark that when $\sf L=\D_4,$ in spite of the apparent symmetry of the $\B_3$'s given in Table (\ref{duquesa}), these correspond to different cases. If one considers the complex algebras $\so(7,\C)$ and $\so(8,\C),$ then the labelling $\begin{dynkinDiagram}D{4}\draw[thick] (root 3) to [out=-45, in=45] (root 4);\end{dynkinDiagram}$ corresponds to the representation $\so(7,\C)\to\so(8,\C)$ that stabilizes a line in $\C^8,$ whilst the labelling $\begin{dynkinDiagram}D{4}\draw[thick] (root 1) to [out=-60, in=180] (root 4);\end{dynkinDiagram}$ corresponds to the fundamental representation of $\so(7,\C)$ associated to the short root of $\B_3.$ This is an irreducible representation with image in $\so(8,\C)$ called \emph{the spin representation}, see Fulton-Harris \cite[Lecture 20, Ex. 20.38]{FultonHarris}.

\begin{figure}
\begin{tikzpicture}

\node[scale=1] (a) at (-4,0){
\begin{tikzcd}[column sep=small]
 & \circ \arrow[d, dash, "\bb"]& & \\ 
 & \circ\arrow[d, dash, "\bb_2"]&  & \\
 & \circ \arrow[dl, dash,swap,"\sroot"] \arrow[dr, dash,"\aa"]  & &\\
  \circ  \arrow[dr, dash,"\aa"] &  & \arrow[dl, dash,"\sroot"] \circ  & \\
  & \circ  \arrow[d, dash,"\bb_{2}"]&   &   \\
  & \circ\arrow[d, dash,"\bb"]& & \\
 & \circ &  &
\end{tikzcd}};
\node[scale=1] at (-4,-4){$\hasse{\D_4}{\peso_\bb}$};

\node[scale=1] at (-4,-5){\dynkin[labels={\bb,\bb_2,\sroot,\aa},scale=1.4] D4};

\node[scale=1] (b) at (-.5,0){\begin{tikzcd}[column sep=small]
 & \circ \arrow[d, "\bb"]& & \\ 
  & \circ \arrow[d, "\bb_2"]& & \\
 & \circ\arrow[d, "\aa"]&  & \\
 & \circ \arrow[d, "\aa"]& & \\ 
  & \circ \arrow[d, "\bb_2"]& & \\
 & \circ\arrow[d, "\bb"]&  & \\
 & \circ  & &
\end{tikzcd}
};
\node[scale=1] at (-.5,-4){$\hasse{\B_3}{\peso_\bb}$};

\node[scale=1] at (-.5,-5){\dynkin[labels={\bb,\bb_2,\aa},scale=1.4] B3};

\node[scale=1] at (-2,.5) {$\begin{dynkinDiagram}D{4}\draw[thick] (root 3) to [out=-60, in=60] (root 4);\end{dynkinDiagram}$}; 
\draw[scale=.4, ->] (a) -- (b);

\node[scale=1] (c) at (2,0){
\begin{tikzcd}[column sep=small]
 & \circ \arrow[d, dash, "\bb"]& & \\ 
 & \circ\arrow[d, dash, "\bb_2"]&  & \\
 & \circ \arrow[dl, dash,swap,"\sroot"] \arrow[dr, dash,"\aa"]  & &\\
  \circ  \arrow[dr, dash,"\aa"] &  & \arrow[dl, dash,"\sroot"] \circ  & \\
  & \circ  \arrow[d, dash,"\bb_{2}"]&   &   \\
  & \circ\arrow[d, dash,"\bb"]& & \\
 & \circ &  &
\end{tikzcd}};
\node[scale=1] at (2,-4){$\hasse{\D_4}{\peso_\bb}$};

\node[scale=1] at (2,-5){\dynkin[labels={\bb,\bb_2,\sroot,\aa},scale=1.4] D4};

\node[scale=1] (d) at (6,0){
\begin{tikzcd}[column sep=small]
 & \circ \arrow[d, dash, "\aa"]& & \\ 
 & \circ\arrow[d, dash, "\bb_2"]&  & \\
 & \circ \arrow[dl, dash,swap,"\aa"] \arrow[dr, dash,"\bb"]  & &\\
  \circ  \arrow[dr, dash,"\bb"] &  & \arrow[dl, dash,"\aa"] \circ  & \\
  & \circ  \arrow[d, dash,"\bb_{2}"]&   &   \\
  & \circ\arrow[d, dash,"\aa"]& & \\
 & \circ &  &
\end{tikzcd}};
\node[scale=1] at (6,-4){$\hasse{\B_3}{\peso_\aa}$};

\node[scale=1] at (6,-5){\dynkin[labels={\bb,\bb_2,\aa},scale=1.4] B3};

\node[scale=1] at (4,0.5) {$\begin{dynkinDiagram}D{4}\draw[thick] (root 1) to [out=-60, in=180] (root 4);\end{dynkinDiagram}$}; 
\draw[scale=1, ->] (c) -- (d);

\end{tikzpicture}\caption{The surjective maps $\hasse{\D_{4}}{\peso_\bb}\to\hasse{\B_3}{\peso_\bb}$ and $\hasse{\D_{4}}{\peso_\bb}\to\hasse{\B_3}{\peso_\aa}$}\label{d4b3}
\end{figure}

\section{Discrete subgroups satisfying a coherence condition w.r.t. eigenspaces}\label{discretegroups}

\subsection{Review on Lie group representations}\label{map}

Let $\sf G$ be a reductive real algebraic Lie group. If $\bar\phi:\sf G\to\GL(V)$ is a rational representation then we denote by $\phi:\ge\to\gl(V)$ the induced representation on its Lie algebra and we speak indistinctively of highest restricted weight, restricted weight spaces, etc of $\phi$ and $\bar\phi.$

One has the following proposition from Tits \cite[Theorem 7.2]{tits} that guarantees existence of representations of $\sf G,$ the reader may also check Abels-Margulis-Soifer \cite[Theorem 6.3]{ABS} We say that $\phi$ is \emph{proximal} if $\dim V^+=1$ (recall Equation \eqref{V+}).

\begin{prop}[Tits \cite{tits}]\label{FundTits} For every $\sroot\in\simple$ there exists an irreducible proximal representation of $\sf G$ whose highest restricted weight is $l\peso_\sroot$ for some $l\in\Z_{\geq1}.$ If $\ge$ is split then one can choose $l=1.$\end{prop}

\begin{defi}For each $\sroot\in\simple$, we will fix and denote by $\Fund_\sroot:\sf G\to\GL(V_\sroot)$ a representation given by the above proposition.\end{defi}

Recall the definition of root spaces $\ge_\aa$ from \S\,\ref{remainders}. For $\aa\in\root^+$ we let $\wk{\ge}_\aa=\ge_{-\aa}$,  $\wk\n  = \bigoplus_{\aa\in\root^+}\wk\ge_{\aa}$ and we consider the opposite minimal parabolic subalgebras $\b =\ge_0\oplus\n$ and $\wk\b=\ge_0\oplus\wk{\n}.$ The \emph{minimal parabolic subgroups} are denoted by $\Borel$ and $\wk{\Borel}$ and defined as the normalizers in $\sf G$ of $\b$ and $\wk\b$ respectively. The groups $\Borel$ and $\wk{\Borel}$ are conjugated. The \emph{complete flag space of $\sf G$} is defined by $\cal F=\sf G/\Borel.$ The $\sf G$-orbit of $$\big([\Borel],[\wk{\Borel}]\big)\in\cal F\times\cal F$$ is the unique open orbit of $\sf G$ and is denoted by $\posgen.$

If $(\fund,V)$ is a proximal irreducible representation, then we let $\vt\subset\simple$ be the set of simple roots non-orthogonal to $\chi_\phi$, $$\vt=\{\sroot\in\simple:\<\chi_\phi,\sroot\>\neq0\}.$$ Consider also  the parabolic subgroup $\sf P_\vt$  whose Lie algebra is defined by $$\frak p_\vt  =\bigoplus_{\sroot\in\root^+\cup\{0\}}\frak g_{\sroot} \oplus\bigoplus_{\sroot\in\<\simple-\vt\>}\frak g_{-\sroot}.$$ The group $\sf P_\vt$ is the stabilizer in $\sf G$ of the line $V^+$.


\begin{defi}\label{type} We will say that $\vt$, or $\sf P_\vt$, is the \emph{type} of the stabilizer of $V^+
$.\end{defi}

We also consider an opposite parabolic subgroup $\wk{\sf P}_\vt$ whose Lie algebra is $$\wk{\frak p}_\vt =\bigoplus_{\sroot\in\root^+\cup\{0\}}\frak g_{-\sroot} \oplus\bigoplus_{\sroot\in\<\simple-\vt\>}\frak g_{\sroot}.$$ It is conjugated to the parabolic group $\sf P_{\ii\vt}.$ We denote the \emph{flag space} associated to $\vt$ by $\cal F_\vt=\sf G/\sf P_\vt.$ The $\sf G$ orbit of the pair $([\sf P_{\vt}],[\check{\sf P}_{\vt}])$ is the unique open orbit for the action of $\sf G$ in the product $\cal F_\vt\times\cal F_{\ii\vt}$ and is denoted by $\posgen_\vt.$

One has a $\Fund$-equivariant algebraic map $$\mapa=\mapa_{\Fund}:\cal F_\vt\to\P(V)$$ defined by $\mapa_{\Fund}(g[\sf P_\vt])=\Fund(g)V^+.$ The $\fund(\a)$-invariant complement $$V^-:=\bigoplus_{\chi\in\poids(\phi)-\{\chi_\phi\}}V^\chi$$ is stabilized by $\wk{\sf P}_\t,$ giving also a map $\wk{\mapa}=\wk{\mapa}_{\Fund}:\cal F_{\ii\vt}\to\P(V^*)$ defined by $\wk\mapa(g\cdot[\wk{\sf P}_\vt])=\Fund(g)V^-$, where  we have used the natural identification between $\P(V^*)$ and $\mathrm{Gr}_{\dim V-1}(V)$ given by $\R\varphi\mapsto\ker\varphi$.

\subsection{Jordan-Kostant-Lyapunov's projection and Benoist's limit cone}

Recall that every element $h\in\sf G$ can be uniquely written as a commuting product $h=h_eh_{ss}h_n$ where $h_e$ is conjugate to an element in $\sf K,$ $h_{ss}$ is conjugate to an element in $\exp(\a^+)$ and $h_n$ is unipotent. The \emph{Jordan-Kostant-Lyapunov projection} $\lambda=\lambda_{\sf G}:\sf G\to\a^+$ is defined such that $h_{ss}$ is conjugated to $\exp\big(\lambda(h)\big).$

If $\grupo\subset\sf G$ is a discrete subgroup, then its \emph{limit cone} is denoted by $\Bcone_\grupo$ and is defined as the smallest closed cone that contains $\{\lambda(g):g\in\grupo\}.$ One has the following fundamental result by Benoist. Recall that $\a_{ss}=\a\cap\ge_{ss}.$

\begin{thm}[{Benoist \cite[Théorème 1.2]{limite}}]\label{Bcono} Let $\grupo<\sf G$ be a Zariski dense subgroup. Then the limit cone $\Bcone_\grupo$ is convex and the intersection $\Bcone_\grupo\cap\a_{ss}$ has non-empty interior in $\a_{ss}.$
\end{thm}

\subsection{Coherent subgroups} For $g\in\GL_d(\R)$ let us denote by $$\lambda(g)=\big(\lambda_1(g),\cdots,\lambda_d(g)\big)\in\a^+$$ its Jordan projection. By definition, the coordinates of $\lambda(g)$ are the logarithms of the modulus of the eigenvalues of $g$, counted with multiplicity and in decreasing order. If $\lambda_1(g)>\lambda_2(g)$ we say that $g$ is \emph{proximal}. Equivalently, the generalized eigenspace associated to the greatest (in modulus) eigenvalue of $g$ is $1$-dimensional. We will denote by $g_+\in\P(\R^d)$ this attracting eigenline and by $g_-$ its $g$-invariant complementary subspace.

A discrete subgroup $\grupo<\PGL_d(\R)$ is \emph{proximal} if it contains a proximal element. One defines then its \emph{limit set} by $$\Lim\grupo^\P=\overline{\{g_+:g\in\grupo\textrm{ proximal}\}}.$$

Recall from the introduction that $\Lim\grupo^\P$ is \emph{minimal} if the only closed $\grupo$-invariant subsets of $\Lim\grupo^\P$ are $\{\emptyset,\Lim\grupo^\P\}.$

\begin{lemma}\label{red-irr}Let $\grupo<\PGL_d(\R)$ be proximal with minimal $\Lim\grupo^\P.$ If $\grupo$ acts totally reducibly in $\R^d$ then $\spa\Lim\grupo^\P$ is an irreducible factor of $\grupo.$\end{lemma}

\begin{proof} Let $g\in\grupo$ be proximal and $V$ an irreducible factor. If $v\in V$ does not lie in $g_-$ then $g^n(\R\cdot v)\to g_+.$ Consequently, since $V$ is closed and $g$-invariant, if $g_+\notin V$ one concludes $V\subset g_-.$ Thus, $g_+$ necessarily belongs to an irreducible factor of $\grupo,$ $W$ say. The subset $\Lim\grupo^\P\cap\P(W)$ is then non-empty, closed and $\grupo$-invariant. Minimality completes the proof.\end{proof} 

\begin{defi}\label{coherente} A discrete subgroup $\grupo<\PGL_d(\R)$ is \emph{coherent} if \begin{itemize}\item[-]there exists a proximal $g_0\in\grupo$ such that $\wedge^2g_0$ is proximal and the eigenline associated to $\lambda_2(g_0)$ belongs to $\spa\Lim\grupo^\P,$
\item[-] the limit sets $\Lim\grupo^\P$ and $\Lim{\wedge^2\grupo}^\P$ are minimal.\end{itemize}\end{defi}

\begin{ex} The typical example of a coherent group are the so-called $(1,1,2)$-hyperconvex representations from Pozzetti-S.-Wienhard \cite{PSW1}.
\end{ex}

The main feature of coherence one should keep in mind is that, necessarily, the generalized eigenspace $V_2(g_0)$ of $g_0$ associated to $\lambda_2(g_0)$ is one dimensional, and both lines $(g_0)_+$ and 
$V_2(g_0)$ lie in the same irreducible factor of $\grupo$ on $\R^d$. This will be further explained in the proof of the following Lemma.

\begin{lemma}\label{raiz1}Let $\grupo<\PGL_d(\R)$ be a coherent subgroup with reductive Zariski closure $\sf H$ and let $\h=\lie(\sf H)$. Then there exists a unique $\sroot\in\simple_\h$ such that for every $g\in\grupo$ one has $$\sroot\big(\jordan_{\sf H}(g)\big)=\lambda_1(g)-\lambda_2(g).$$ Moreover $\dim \h_\sroot=1$.
\end{lemma}

\begin{proof}By Lemma \ref{red-irr} the representations $\sf H|\spa\Lim\grupo^\P$ and $\sf H|\spa\Lim{\wedge^2\grupo}^\P$ are irreducible. Let $\chi_1$ and $\chi_2$ be their highest restricted weights, then $2\chi_1-\chi_2$ verifies that for all $g\in\grupo$ one has $2\chi_1-\chi_2(\jordan_{\sf H}(g))=\lambda_1(g)-\lambda_2(g).$

Denote by $\{W_i\}_1^k$ the irreducible factors of $\grupo$ enumerated so that  $W_1=\spa\Lim\grupo^\P$. For $g\in\grupo$ with $\wedge^2g$ proximal, denote by $V_2(g)$ either the eigenline associated to $\lambda_2(g)$ if $g$ is proximal, or the $2$-dimensional Jordan block associated to $\lambda_1(g)$ otherwise. One readily sees that, in both situations, the vector space $V_2(g)$ necessarily intersects one of the $W_i$'s. 

We can identify $\Lim{\wedge^2\grupo}^\P$ as a subset of $\grass_2(\R^d)$ and thus consider the closed $\wedge^2\grupo$-invariant subsets $$\LL_i=\big\{P\in\Lim{\wedge^2\grupo}^\P:P\cap W_i\neq\{0\}\big\}.$$The intersections $\LL_i\cap\LL_j$ are also invariant and closed so by minimality, each intersection is either empty or $\Lim{\wedge^2\grupo}^\P.$ However, the element $g_0$ from the definition of coherence is proximal with $\wedge^2g_0$ proximal, so its attracting line $(\wedge^2g_0)_+\in\Lim{\wedge^2\grupo}^\P$  is  $g_0\oplus V_2(g_0) \in\grass_2(\R^d)$. This latter plane is, by assumption, contained in $W_1=\spa\Lim\grupo^\P$, which yields \begin{itemize}\item[-] $\LL_1=\Lim{\wedge^2\grupo}$ and \item[-] all intersections $\LL_1\cap\LL_j$, for $j>1$, are empty.\end{itemize}
We conclude that $V_2(g)\subset\spa\Lim\grupo^\P$ for every $g\in\grupo$ with proximal $\wedge^2g.$

Applying \S\ref{reps} to $\sf H|\spa\Lim\grupo^\P$ together with the preceding paragraph, one has that for every $g\in\grupo$ there exists $\aa_g\in\simple_\h$ such that $\aa_g(\lambda(g))=\lambda_1(g)-\lambda_2(g).$ Since the limit cone $\Bcone_\grupo$ has non-empty interior on $\a\cap\h_{ss},$ (Benoist's Theorem \ref{Bcono}) and $\simple_\h$ is a finite set, there exists an open sub-cone $\scr C\subset\Bcone_\grupo$ and a root $\sroot\in\simple_\h$ such that for every $v\in\scr C$ $$\sroot(v)=\big(2\chi_1-\chi_2\big)(v).$$ Since both functions are linear and coincide on an open set, they must coincide and $\sroot$ is the required root. The same argument gives uniqueness of $\sroot$. The fact that $\h_\sroot$ is one dimensional follows from the fact that, for every $g\in\grupo$, up to conjugation, one has $\h_\sroot V^+=V^{\chi_1-\sroot}$ is the eigenspace associated to $\lambda_2(g_0)$, which is one dimensional.\end{proof}

\begin{defi}  Let $\sf G$ be a reductive group and $\grupo$ a discrete subgroup. Then $\grupo$ is \emph{totally coherent} if for every $\sroot\in\simple$ the subgroup $\Fund_\sroot(\grupo)$ is coherent.
\end{defi}

The following is the main result of this section.

\begin{prop}\label{existenMapas} Let $\sf G$ be a real-algebraic simple group and $\grupo<\sf G$ a totally coherent discrete subgroup with reductive Zariski closure $\sf H.$  Then $\h_{ss}$ is simple split. Moreover, there exists a surjective function $f:\simple_\ge\to\simple_\h$ and, for every $\aa\in\simple_\ge,$ a surjective map with labeling $f$ between the diagrams $$\T^f:\hasse\ge{\ell_\aa\peso_\aa}\to\hasse\h{n_\aa\peso_{ f(\aa)}},$$ for some $n_\aa\in\Z_{\geq1}.$ If $\aa$ is extremal then $f(\aa)$ is extremal, if moreover $\rk \h_{ss}>1$, $2\aa\notin\root$ and $\ell_\aa=1$ then $n_\aa=1.$\end{prop}

\begin{proof} Let us denote by $\bar\iota:\sf H\to\sf G$ the representation induced by the inclusion of $\sf H$ in $\sf G$ and by $\iota:\h\to\ge$ its derivative.

Since $\grupo$ is totally coherent, applying Lemma \ref{raiz1} to each representation $\Fund_\sroot$ of $\sf G$ provides a function $f:\simple_\ge\to\simple_\h$ such that for every $g\in\grupo$ and $\sroot\in\simple_\ge$ one has \begin{equation}\label{raicesiguales0}f(\sroot)\big(\jordan_{\sf H}(g)\big)=\sroot\big(\jordan_{\sf G}(\bar{\iota}(g))\big).\end{equation}

Consider then $\aa\in\simple_\ge$ and the associated fundamental representation $\Fund_\aa:\sf G\to\GL(V).$ Since $\sf H$ is reductive, Lemma \ref{red-irr} implies that $W=\spa\Lim{\Fund_\aa\grupo}^\P$ is an irreducible factor of $\Fund_\aa\sf H.$ Let $\phi:\h\to\frak{gl}(W)$ be the representation of $\h$ defined by $\phi=\fund_\aa(\iota\h)|W$ and $\chi_\phi\in\poids_\h(\phi)$ its highest restricted weight.

As stated in Remark \ref{generador} every element $\chi\in\poids_\ge(\fund_\aa)$ is of the form \begin{equation}\label{pesosiguales0}\chi=\ell_\aa\peso_\aa-\sum_{\sroot\in\simple_\ge}k_\sroot\sroot,\end{equation} where $k_\sroot\in\Z_{\geq0}$ for every $\sroot$. Define then function $\T^f:\poids_\ge(\fund_\aa)\to\poids_\h$  by $$\T^f(\chi)=\chi_\phi-\sum_{\sroot\in\simple_\ge}k_\sroot f(\sroot),$$ if $\chi$ is as in Equation (\ref{pesosiguales0}). For every $\chi\in\poids_\ge(\fund_\aa)$ and $\bb\in\root_\h$ one has $$\<\T^f(\chi),\bb\>=\<\chi_\phi,\bb\>-\sum_{\sroot\in\simple_\ge} k_\sroot\<f(\sroot),\bb\>\in\Z,$$ so $\T^f(\chi)$ is indeed a weight of $\h$, moreover  $\T^f$ is level preserving. Observe also that for every $g\in\grupo$ one has, by Equation (\ref{raicesiguales0}), that $$\T^f(\chi)\big(\jordan_{\sf H}(g)\big)=\chi\big(\jordan_{\sf G}(\bar{\iota} g)\big),$$ so that for every $v\in\Bcone_\grupo$ one has $\T^f(\chi)(v)=\chi(\iota v).$ Thus, for every $v\in\Bcone_\grupo$ and $w\in V^{\T^f(\chi)}$ one has $$\phi(v)w=\chi(\iota v)w=\big(\T^f(\chi)(v)\big)w.$$

Since $\Bcone_\grupo$ has non-empty interior in $\a_{\h,ss}$ (Theorem \ref{Bcono}) and $\poids_\ge(\fund_\aa)$ is finite, there exists an open sub-cone $\scr C\subset\Bcone_\grupo$ such that for every $u\in\scr C$ the eigenvalues $\T^f(\chi)(u)$, for $\chi\in\poids_\ge(\fund_\aa)$, are pairwise distinct. This is to say, the decomposition $$V=\bigoplus_{\chi\in\poids_\ge(\fund_\aa)}V^{\T^f(\chi)}$$ consists, for every $u\in\scr C$, on eigenspaces associated to pairwise distinct eigenvalues of $\phi(u)$. Thus, intersecting with $W$ and since $\scr C$ is open, we obtain that $$W=\bigoplus_{\chi\in\poids_\ge(\phi_\aa)}W\cap V^{\T^f(\chi)}$$ is the weight space decomposition of $\phi$. Thus $\T^f$ has values in $\poids_\h(\phi)$ and is moreover surjective onto this set.

Since $\sf G$ is simple, $\fund_\aa$ is injective and thus, since any weight of $\fund_\aa(\iota\h)$ is contained in $\poids_\h(\phi),$ $\fund_\aa(\iota\h_{ss})$ is simple and thus $\h_{ss}$ is. Consequently, $f$ is surjective and, since $\dim (\h_{ss})_{f(\aa)}=1$ for every $\aa$ (Lemma \ref{raiz1}), $\h_{ss}$ is split.


From surjectivity of $\T^f,$ and since there is only one weight of $\fund_\aa$ of level $2$ (the weight $\ell_\aa\peso_\aa-\aa$, recall Example \ref{48}) one has that for every $\bb\in\simple_\h-\{f(\aa)\}$ the linear form $\chi_\phi-\bb$ is not a weight, hence $\<\chi_\phi,\bb\>=0$ and thus $\chi_\phi=n_\aa\peso_{f(\aa)}$ for some $n_\aa\in\Z_{\geq1}.$

Let us assume from now on that $\aa$ is an extremal root of $\simple_\ge,$ so that the only weights of level 3 of $\fund_\aa$ are $\ell_\aa\peso_\aa-\aa-\bb$ for a unique root $\bb\in\simple_\ge,$ and  $\ell_\aa\peso_\aa-2\aa$ (only if $\ell_\aa\geq2$ or if $2\aa\in\root$). This implies that the only weights of level 3 of $\phi$ are $n_\aa\peso_{f(\aa)}-f(\aa)-f(\bb),$ and possibly $n_\aa\peso_{f(\aa)}-2f(\aa).$ 

Hence $\<n_\aa\peso_{f(\aa)}-f(\aa),\sroot\>=0$ for every $\sroot\in\simple_\h-\{f(\aa),f(\bb)\}$ from which $f(\aa)$ is an extremal root of $\simple_\h.$ Moreover, either \begin{itemize}\item[-]$f(\aa)=f(\bb)$ i.e. for every $\sroot\in\simple_\h-\{f(\aa)\}$ one has $$0=\<n_\aa\peso_{f(\aa)}-f(\aa),\sroot\>=-\<f(\aa),\sroot\>$$ and thus $\h_{ss}$ has rank $1,$\item[-] or $ f(\aa)\neq f(\bb).$ In this case, if one assumes moreover that $\ell_\aa=1$ and $2\aa\notin\root$, then $n_\aa\peso_{f(\aa)}-2f(\aa)\notin\poids_\h(\phi)$ and hence $n_\aa=1.$\end{itemize}This completes the proof. \end{proof}

\subsection{Classification of Zariski closures of totally coherent groups}\label{clasif}

Throughout this section, $\ge$ is a simple split real Lie algebra, $\sf G$ is a real-algebraic Zariski connected Lie group with Lie algebra $\ge$ and $\grupo<\sf G$ is a totally coherent discrete subgroup with reductive Zariski closure $\sf H.$ The purpose is to classify the pairs $(\h_{ss},\phi)$ where $\phi:\h_{ss}\to\ge$ is the representation induced by the inclusion $\sf H\subset\sf G.$ By Proposition \ref{existenMapas} $\h_{ss}$ is simple split.

One begins by the following:

\begin{cor}\label{rg=1}If $\h_{ss}$ has rank $1$ then it is a principal  $\sl_2(\R)$ of $\ge.$
\end{cor}

\begin{proof} Consider $\aa\in\simple$ and let us compose the inclusion of $\h_{ss}$ with a fundamental representation $\fund_\aa$ of $\ge$. Throughout the proof of Proposition \ref{existenMapas} it is stablished that the highest weight space $V^+$ of $\fund_\aa$ is also the highest weight space of some non-trivial irreducible factor of $\fund_\aa(\iota\h_{ss}),$ of highest weight $n_\aa\peso_{f(\aa)}$, for some $n_\aa\in\Z_{\geq1}$, and a function $f:\simple_\ge\to\simple_{\h}$, necessarily constant in this is case.

There exists then a non-zero $\sf f\in\iota(\h_{ss})\cap\wk{\frak n}$. Consider $w\in V^+=V^{n_\aa\peso_{f(\aa)}}$. Since $V^{n_\aa\peso_{f(\aa)}}$ is the highest weight space of the representation $\phi:\h_{ss}\to\gl(W)$, $$0\neq\phi(\sf f)v\in V^{n_\aa\peso_{f(\aa)}-f(\aa)}\subset V^{\ell_\aa\peso_\aa-\aa}.$$

Additionally, we compute $\phi_\aa(\sf f)v$ upon writing $\sf f=\sum_{\sroot\in\root^+}b_\sroot\yy_\sroot$. To this end, consider the set $\sf R_\aa$ of positive roots with non-vanishing coefficient on $\aa$ (in the basis $\simple$)  and its complement $\sf R_\aa^\complement$ on $\root^+$, \begin{alignat*}{2}\sf R_\aa & =\{\bb\in\root^+: \<\peso_\aa,\bb\>\neq0\}\\ \sf R_\aa^\complement & =\{\bb\in\root^+: \<\peso_\aa,\bb\>=0\}\end{alignat*} By Proposition \ref{cuerda}, if $\bb\in\sf R_\aa^\complement$ one has $\ell_\aa\peso_\aa-\bb\notin\poids(\phi_\aa)$, so $\phi_\aa(\sf y_\bb) v=0$. However, again by Proposition \ref{cuerda}, if $\bb\in\sf R_\aa$ then $\ell_\aa\peso_\aa-\bb\in\poids(\phi_\aa)$, giving $\phi_\aa(\sf y_\bb) v\neq0$.

Thus, \begin{alignat*}{2}V^{\ell_\aa\peso_\aa-\aa}\ni\phi_\aa(\sf f)v & =\sum_{\sroot\in\root^+}b_\sroot\big(\phi_\aa(\yy_\sroot)v\big)\\ & =\sum_{\sroot\in\sf R_\aa}b_\sroot\big(\phi_\aa(\yy_\sroot)v\big)+\sum_{\sroot\in\sf R_\aa^\complement}b_\sroot\big(\phi_\aa(\yy_\sroot)v\big)\\ & = \sum_{\sroot\in\sf R_\aa}b_\sroot\big(\phi_\aa(\yy_\sroot)v\big).\end{alignat*} Since the weight spaces $V^{\ell_\aa\peso_\aa-\bb}$ are in direct sum for distinct $\bb\in\root_\aa$, one concludes $b_\bb=0$ for all $\bb\in\sf R_\aa-\{\aa\}$ and $b_\aa\neq0$. 

The same argument applied to the remaining fundamental representations $\phi_\sroot$, for $\sroot\in\simple$, give that $\sf f=\sum_{\sroot\in\simple}b_\sroot\sroot$ and that $b_\sroot\neq0$ for all $\sroot\in\simple$. Kostant's Theorem \ref{nilpoPrinc} asserts then that $\h_{ss}$ is a principal $\sl_2(\R).$\end{proof}
%
%
%
%
%

If the rank of $\h_{ss}$ is at least $2$ then, since the fundamental representations of $\ge$ verify $\ell_\aa=1$ for all $\aa\in\simple_\ge,$ Proposition \ref{existenMapas} provides a surjective function $f:\simple_\ge\to\simple_\h$ such that the image of an extremal root is an extremal root, and for every $\aa\in\simple_\ge$ a surjective map $\T^f:\hasse\ge{\peso_\aa}\to\hasse\h{\peso_{f(\aa)}}$ between the corresponding Hasse diagrams. Applying the Table (\ref{duquesa}) given by Lemma \ref{diagramasClasif} one concludes at once the following Corollary.

\begin{cor}\label{rg=r} If $\rk \h_{ss}\geq2$ and $\h_{ss}\neq\ge,$ then the only possibilities for $\phi:\h_{ss}\to\ge$ are, up to $\Int\ge$-conjugation, the ones listed in Table \ref{tricota}.
\end{cor}

\begin{table}[h!]
  \begin{center}
    \begin{tabular}{r|r|r} 
      $\ge$ & $\h_{ss}$ & $\phi:\h_{ss}\to\ge$  \\
      \hline
     
      \multirow{2}{*}{$\sl_{2n+1}(\R)$} & $\so(n,n+1)$ $\forall n$ & defining representation \\
       & $\ge_2$ if $n=3$ & fundamental for the short root\\\hline
       $\sl_{2n}(\R)$ & $\sp(2n,\R)$ & defining representation \\\hline
      $\so(3,4)$ & $\ge_2$ & fundamental for the short root \\\hline
      \multirow{4}{*}{$\so(n,n)$} & $\so(n-1,n)$ $\forall n\geq3$ & stabilizer of a non-isotropic line \\
       & $\so(3,4)$ if $n=4$ & fundamental for the short root \\\cline{3-3}
       & \multirow{2}{*}{$\ge_2$ if $n=4$} & stabilizes a non-isotropic line $L$ and is  \\ & & fundamental for the short root on $L^\perp$\\\hline 
      $\e_6$ & $\f_4$ & $\Fix(\inv_0)$ (Example \ref{Inv})
    \end{tabular}
    \caption{Statement of Corollary \ref{rg=r}} \label{tricota}
  \end{center}
\end{table}

\section{Total positivity}\label{total>0}

Throughout this section $\sf G$ denotes the real points of a Zariski connected real-algebraic simple split group.

\subsection{Lusztig's total positivity}\label{>0}Let us fix, for each simple root $\sroot\in\simple,$ algebraic group isomorphisms $\x_\sroot:\R\to\exp\ge_\sroot,$ $y_\sroot:\R\to\exp\wk{\ge}_\sroot$ and $\hh_\sroot:\R\to\exp (\R\cdot h_\sroot)$ so that $$(\begin{smallmatrix}1 & t \\ 0 & 1\end{smallmatrix})\mapsto \x_\sroot(t),\\ \, (\begin{smallmatrix}1 & 0 \\ t & 1\end{smallmatrix})\mapsto \y_\sroot(t),\\ \, (\begin{smallmatrix}t & 0 \\ 0 & t^{-1}\end{smallmatrix})\mapsto \hh_\sroot(t),$$ defines a morphism $\SL_2(\R)\to\sf G$ (recall \S\,\ref{xalpha} on the existence of such morphisms). The collection $\pin=\big\{x_\sroot:\sroot\in\simple\}$ is called a \emph{pinning} of $\sf G$ and two pinnings are conjugated by $\sf G.$


Let $\wk\n=\bigoplus_{\aa\in\root^+}\wk{\ge}_\aa$. Denote by $U=\exp\n$ and by $\wk{U}=\exp\wk{\n}$ the unipotent radicals of $\Borel$ and $\wk{\Borel}$ respectively. Let $A=\exp\a$ and let $\sf M$ be the centralizer in $\sf K$ of $\exp\a,$ one has \begin{equation}\label{uni}B=\sf M A U.\end{equation}



Let $w_0\in\Weyl$ be the longest element and consider a reduced expression $w_0=r_N\cdots r_1$ as a product of reflections associated to simple roots. Let us denote, for each $r_i$ the associated simple root by $\sigma_{r_i}\in\simple.$ The number $N$ equals $|\root^+|,$ but we will not require this fact.

Consider the maps $\Psi^\pin:(\R_{>0})^N\to U$ and $\wk{\Psi}^\pin:(\R_{>0})^N\to \wk{U}$ defined by \begin{alignat}{2}\Psi^\pin(a_1,\cdots,a_N) & =\x_{\sroot_{r_N}}(a_N)\cdots\x_{\sroot_{r_1}}(a_1),\nonumber \\ \wk{\Psi}^\pin(a_1,\cdots,a_N) & =\y_{\sroot_{r_N}}(a_N)\cdots\y_{\sroot_{r_1}}(a_1).\end{alignat}

We summarize several results from {Lusztig \cite[\S2]{Lusztig-TP}} in the following theorem.

\begin{thm}[{Luzstig \cite[\S2]{Lusztig-TP}}]The images $U_{>0}=\Psi^\pin\big((\R_{>0})^N\big)$ and $\wk U_{>0}=\wk{\Psi}^\pin\big((\R_{>0})^N\big)$ are semi-groups independent of the chosen reduced expression of $w_0.$ The product $$\sf G_{>0}=\wk{U}_{>0}A U_{>0}= U_{>0}A\wk{U}_{>0}$$ is also a semi-group and every element $g\in \sf G_{>0}$ has a unique expression of the form $g=\wk u t v$ with  $\wk u\in\wk{U}_{>0},$ $t\in A$ and $v\in U_{>0}.$
\end{thm}

Even though we omit the pinning notation on the semi-groups $U_{>0},$ $\wk{U}_{>0}$ and $\sf G_{>0},$ they do depend on the pinning $\pin.$ For example, fixing the pinning on $\SL_n(\R)$ $$\big(\begin{smallmatrix} 1 & t \\ 0 & 1\end{smallmatrix}\big)\mapsto x_i(t)=\id +te_{i,i+1},\quad i\in\lb1,d-1\rb,$$ where $e_{i,j}$ is the $n\times n$ matrix consisting of vanishing entries except at $(i,j)$, whose entry equals $1$, gives the semi-group $\SL_n(\R)_{>0}$ of totally positive matrices mentioned in the Introduction. However one may consider other pinnings. We list below four possibilities for the unipotent semi-groups $U_{>0}$ in $\SL_3(\R)$ corresponding to different pinnings: \begin{alignat*}{2} U^{\pin_1}_{>0} & =\Big\{\left(\begin{smallmatrix} 1 & x+z & xy \\ & 1 & y \\ & &1\end{smallmatrix}\right):x,y,z\in\R_{>0}\Big\},\\ U^{\pin_2}_{>0} & =\Big\{\left(\begin{smallmatrix} 1 & -(x+z) & -xy \\ & 1 & y \\ & &1\end{smallmatrix}\right):x,y,z\in\R_{>0}\Big\},\\ U^{\pin_3}_{>0} & =\Big\{\left(\begin{smallmatrix} 1 & x+z & -xy \\ & 1 & -y \\ & &1\end{smallmatrix}\right):x,y,z\in\R_{>0}\Big\},\\ U^{\pin_4}_{>0} & =\Big\{\left(\begin{smallmatrix} 1 & -(x+z) & xy \\ & 1 & -y \\ & &1\end{smallmatrix}\right):x,y,z\in\R_{>0}\Big\}.\end{alignat*}

\subsection{Positivity of flags}

The positive semi-group $\sf G_{>0}$ determines a special subset $\posF\subset\cal F$ defined by $$\posF=\sf G_{>0}\cdot[\Borel]=\wk{U}_{>0}\cdot[\Borel]=U_{>0}\cdot[\wk{\Borel}].$$ 

Let us say that an ordered triple $(x_1,x_2,x_3)\in\cal F^3$ is \emph{in general position} if for all $1\leq i<j\leq 3$ one has $(x_i,x_j)\in\posgen.$ Then one has the following.

\begin{prop}[{Lusztig \cite[Prop. 8.14]{Lusztig-TP}}]\label{abierto} The subset $\posF$ is a connected component of $$\Big\{x\in\cal F:\big([\Borel],x,[\wk\Borel]\big)\textrm{ is in general position}\Big\}.$$ In particular it is an open subset of $\cal F.$
\end{prop}

One then defines positivity on triples of flags as being $\sf G$-equivariant, consequently the notion will not depend on the pinning:

\begin{defi} A triple of flags in general position $(x,y,z)$ is \emph{positive} if there exists $g\in \sf G$ such that $g x=[\Borel], gz=[\wk\Borel]$ and $gy\in\posF.$
\end{defi}

\subsection{Simply laced $\sf G$}\label{laced}

Recall that $\ge$ is \emph{simply laced} if for every pair $\sroot,\aa\in\simple$ one has $\<\sroot,\aa\>=\<\aa,\sroot\>$ (recall the definition of $\<\,,\>$ in Equation \eqref{prodsimply}). Equivalently, the Dynkin diagram of $\ge$ does not contain a double or triple arrow. Moreover, $\sf G$ is called \emph{algebraically simply connected} if every finite covering from a real algebraic group onto $\sf G$ is trivial, equivalently the group $\sf G_\C$ of $\C$-points of $\sf G$ is simply connected in the topological sense.

\begin{prop}[{Lusztig \cite[\S 3.1 and Prop. 3.2]{Lusztig-TP}}]\label{>0pesos} Assume that $\sf G$ is simply laced and algebraically simply connected. Let $\Fund:\sf G\to\GL(V)$ be an irreducible real representation, then there exists a basis $\base_{\Fund}$ of $V$ such that \begin{itemize}\item[-] each element of $\base_{\Fund}$ is contained in a restricted weight space of $\fund,$ \item[-] for every $g\in \sf G_{>0},$ the map $\Fund(g):V\to V$ has positive entries on the basis $\base_{\Fund}.$\end{itemize}\end{prop}


\subsection{Theorem \ref{tA} for simply laced $\sf G$}

We devote this section to the proof of Theorem \ref{tA} when $\sf G$ is simply laced and algebraically simply connected. We prove that a discrete subgroup verifying the hypothesis of Theorem \ref{tA} is totally coherent.

\begin{cor}\label{grupos->coherentes}Let $\sf G$ be simply laced and algebraically simply connected, and $\grupo$ a subgroup with minimal limit sets and such that $\Lim{\grupo,\simple}$ contains a positive loxodromic triple. Then $\grupo$ is totally coherent.\end{cor}

\begin{proof}For $\sroot\in\simple$ consider the fundamental representation $\Fund_\sroot:\sf G\to\GL(V)$ and the $\Fund_\sroot$-equivariant map $\mapa:\cal F_{\{\sroot\}}\to\P(V)$ from \S\,\ref{map}. Let also $\wedge^2:\GL(V)\to\GL(\wedge^2V)$ denote the second exterior power representation. 

By minimality one has that $$\mapa\big(\Lim{\grupo,\{\sroot\}}\big)=\Lim{\Fund_\sroot(\grupo)}^\P.$$ Moreover, since the only second level weight of $\Fund_\sroot$ is $\peso_\sroot-\sroot$ (recall Example \ref{48}), the representation $\wedge^2\Fund_\sroot$ of $\sf G$ is proximal, though it may be reducible. Denote by $\Wedge:\sf G\to\GL(V')$ the $\sf G$-irreducible factor containing the highest weight of $\wedge^2\Fund_\sroot$. It contains the attracting points of $\wedge^2g$ for every $g\in\sf G$ proximal on $\cal F.$ Let $\t\subset\simple$ be the type of the stabilizer  in $\sf G$ (recall Definition \ref{type}) of $V^{\peso_\sroot}\wedge V^{\peso_\sroot-\sroot}$. The limit set $\Lim{\grupo,\t}$ is also minimal and one has $\mapa_\Wedge\big(\Lim{\grupo,\t}\big)=\Lim{\wedge^2\Fund_\sroot(\grupo)}^\P$ so the latter is thus minimal.

Finally, consider $g_0\in\grupo$ proximal on $\cal F$ and $x_0\in\Lim{\grupo,\simple}$ so that $({g_0}_+,x_0,{g_0}_-)$ is a positive triple. We can assume that ${g_0}_+=[\Borel]$ and ${g_0}_-=[\wk{\Borel}]$ so that $\mapa({g_0}_+)=V^+$ and $\wk{\mapa}({g_0}_-)=V^-$ (recall notation from \S\,\ref{map}). We want to show that $V^{\peso_\sroot-\sroot}$ belongs to $\spa\mapa(\Lim{\grupo,\{\sroot\}}).$

Let $g\in\sf G_{>0}$ be such that $$\mapa(x_0)=\mapa\big(g\cdot[\Borel]\big)=\Fund_\sroot(g)\cdot\mapa({g_0}_+).$$ Consider then the $2$-dimensional subspace $P_{x_0}=\mapa({g_0}_+)\oplus\Fund_\sroot(g)\mapa({g_0}_+)$ and let $\ell_{x_0}\in\P(V)$ be the intersection $$\ell_{x_0}=P_{x_0}\cap V^-.$$ Since $\sf G$ is simply laced, Lusztig's Proposition \ref{>0pesos} applies to give that $\Fund_\sroot(g)$ has positive coefficients in $\base_{\Fund_\sroot}.$ In particular, if $v\in V^+-\{0\}$ the vector $\Fund_\sroot(g)v$ has positive coefficients in $\base_{\Fund_\sroot}.$ The line $\ell_{x_0}$ is thus not contained in any subspace spanned by a partial sum of weights in $\poids(\fund_\sroot)-\{\peso_\sroot\},$ i.e. $\ell_{x_0}$ is not contained in any $\Fund_\sroot(g_0)$-invariant subspace of $V^-.$ Consequently, the sequence $\Fund_\sroot({g_0}^n)\cdot\ell_{x_0}$ approaches, as $n\to+\infty,$ the $\Fund_\sroot({g_0})$-invariant subspace of $V^-$ associated to the top eigenvalue of $\Fund({g_0})|V^-,$ which is $V^{\peso_\sroot-\sroot}.$ This completes the proof.\end{proof}

Corollary \ref{grupos->coherentes} together with subsection \S\,\ref{clasif} give thus the following.

\begin{cor}\label{simplyL}Let $\sf G$ be simply laced and algebraically simply connected, and let $\grupo<\sf G$ have reductive Zariski closure $\sf H$, minimal limit sets and such that $\Lim{\grupo,\simple}$ contains a positive loxodromic triple. Then $\h_{ss}$  is either $\ge,$ a principal $\sl_2(\R)$ or $\Int\ge$-conjugated to the possibilities listed in Table \ref{cuartirolo}.\end{cor}

\begin{table}[h!]
  \begin{center}
    \begin{tabular}{r|r|r} 
      $\ge$ & $\h_{ss}$ & $\phi:\h_{ss}\to\ge$  \\
      \hline
     
      \multirow{2}{*}{$\sl_{2n+1}(\R)$} & $\so(n,n+1)$ $\forall n$ & defining representation \\
       & $\ge_2$ if $n=3$ & fundamental for the short root\\\hline
       $\sl_{2n}(\R)$ & $\sp(2n,\R)$ & defining representation \\\hline
      \multirow{4}{*}{$\so(n,n)$} & $\so(n-1,n)$ $\forall n\geq3$ & stabilizer of a non-isotropic line \\
       & $\so(3,4)$ if $n=4$ & fundamental for the short root \\\cline{3-3}
       & \multirow{2}{*}{$\ge_2$ if $n=4$} & stabilizes a non-isotropic line $L$ and is  \\ & & fundamental for the short root on $L^\perp$\\\hline 
      $\e_6$ & $\f_4$ & $\Fix(\inv_0)$ (Example \ref{Inv})
    \end{tabular}
    \caption{Statement of Corollary \ref{simplyL}} \label{cuartirolo}
  \end{center}
\end{table}

\subsection{Descent} The purpose of this section is to briefly explain how to bypass the simply-laced hypothesis in Corollary \ref{grupos->coherentes}. We use a standard technique called \emph{descent}. It consists on observing that every simple split Lie algebra $\ge$ is the fixed point set of an automorphism $\bar\auto:\dot\ge\to\dot\ge$ of a simply laced split simple Lie algebra $\dot\ge.$ One requires also that the action of $\bar\auto$ on the simple roots of $\dot\ge$ is such that if $\aa,\bb\in\simple_{\dot\ge}$ are in the same $\bar\auto$-orbit then $\<\aa,\bb\>=0.$ See Table \ref{cinque}.

\begin{table}[h!]
  \begin{center}
    \begin{tabular}{r|r|r} 
      type of $\dot\ge$ & type of $\ge$ & orbits of $\bar\auto$  \\\hline
     $\A_{2n-1}$ & $\Ce_{n}$ & $\begin{dynkinDiagram}A{}\draw[thick] (root 1) to [out=-45, in=-135] (root 4);\draw[thick] (root 2) to [out=-45, in=-135] (root 3);\end{dynkinDiagram}$ \\\hline
      $\D_n$ & $\B_{n-1}$ & $\begin{dynkinDiagram}D{}\draw[thick] (root 5) to [out=-45, in=45] (root 6);\end{dynkinDiagram}$ \\\hline
       $\D_4$ & $\Ge_2$  & $\begin{dynkinDiagram}D{4}\draw[thick] (root 1) to [out=-60, in=180] (root 4);\draw[thick] (root 3) to [out=-60, in=60] (root 4);\draw[thick] (root 1) to [out=60, in=180] (root 3);\end{dynkinDiagram}$\\\hline 
      $\EE_6$ & $\F_4$ & $\begin{dynkinDiagram}E6\draw[thick] (root 1) to [out=-45, in=-135] (root 6);\draw[thick] (root 3) to [out=-45, in=-135] (root 5);\end{dynkinDiagram}$
    \end{tabular}
    \caption{} \label{cinque}
  \end{center}
\end{table}

With these considerations, one has the following proposition from Lusztig.

\begin{prop}[{Lusztig \cite[\S 8.8]{Lusztig-TP}}]\label{descent}Let $\sf G$ be algebraically simply connected. Then there exists a simply laced, simply connected, simple split group $\dot{\sf G}$ and a rational representation $\Wedge:\sf G\to\dot{\sf G}$ together with an equivariant map $\mapa:\cal F_{\sf G}\to\cal F_{\dot{\sf G}}$ such that $$\mapa\Big(\big(\cal F_{\sf G}\big)_{>0}\Big)\subset \Big(\big(\cal F_{\dot{\sf G}}\big)_{>0}\Big).$$
\end{prop}

We can now conclude the proof of Theorem \ref{tA}.

\begin{cor}\label{final} Let $\sf G$ be the real points of a real-algebraic, Zariski connected, simple split group. Let $\grupo<\sf G$ be as in Theorem \ref{tA}. Then the semi-simple part $\h_{ss}$ is either $\ge,$ a principal $\sl_2(\R)$ or $\Int\ge$-conjugated to the possibilities listed in Table \ref{tricota}.
\end{cor}

\begin{proof}By passing to a finite cover we can assume that $\sf G$ is simply connected, the pre-image of $\grupo$ under this covering has again minimal limit sets and its limit set on $\cal F$ contains a positive loxodromic triple. From Proposition \ref{descent} one finds a simply-laced $\dot{\sf G}$ and a rational representation $\Wedge:\sf G\to\dot{\sf G}$ such that $\Wedge\grupo$ is partially positive. Applying Corollary \ref{simplyL} to $\Wedge\grupo$ gives the required result.\end{proof}

\subsection{Partially positive representations preserve type}\label{positiveReps} Recall that $\sf G$ is the real points of a real algebraic, Zariski connected, simple split group.

Let $\gripe$ be a proper Gromov-hyperbolic space and $\G<\isom(\gripe)$ a non-elementary discrete subgroup, then one has the following facts from Ghys-de la Harpe \cite[\S 8.2]{ghysharpe}:\begin{itemize}\item[i)] the action of $\G$ on the visual boundary of $X$ has a smallest closed $\G$-invariant subset denoted by $\bord X_\G$, the $\G$-action on $\bord X_\G$ is thus minimal;
\item[ii)] every $\g\in\G$ is either \begin{itemize}\item[-] of finite order (called \emph{elliptic}),\item[-] \emph{proximal}, i.e. has two fixed points $\g_-,\g_+\in \circulo$ such that for every $x\in \circulo-\{\g_-\}$ one has $\g^nx\to\g_+$ as $n\to+\infty,$\item[-] \emph{parabolic}, i.e. has a unique fixed point $x_\g\in\circulo$ and every $x\in\circulo$ converges to $x_\g$ under the iterates $\g^n$ as $n\to+\infty$ (some points will drift away from $x_\g$ before coming back though). \end{itemize}\item[iii)] The attracting points of proximal elements are dense in $\circulo.$
\end{itemize}

Let us fix throughout this subsection a partially positive representation $\rho:\G\to\sf G$ with continuous $\rho$-equivariant map $\xi:\circulo\to\cal F.$ We begin by showing that it is type preserving. Recall from Equation \eqref{uni} the definition of unipotent radical.

\begin{prop}\label{typeGamma}If $\g\in\G$ is proximal then $\rho(\g)$ is proximal on $\cal F$ with attracting flag $\xi(\g_+)$ and repelling flag $\xi(\g_-).$ If $\upupsilon\in\G$ is parabolic then there exists $k\in\N_{\geq1}$ such that $\rho(\upupsilon^k)$ belongs to the unipotent radical of $\xi(x_\upupsilon),$ moreover, there exists an open set $\cal O\subset \cal F$ such that $\upupsilon^nz\to\xi(x_\upupsilon)$ for every $z\in\cal O.$\end{prop}

\begin{proof}We divide the proof into Lemmas \ref{lox} and \ref{parab} below.\end{proof}

Let $\sf M$ be the centralizer in $\sf K$ of $\exp\a,$ as $\ge$ is split this is a finite group. For $\sroot\in\simple,$ let us denote by $\Fund=\Fund_\sroot:\sf G\to\SL(V)$ and by $\mapa:\cal F\to\P(V),$ $\wk{\mapa}:\cal F\to\P(V^*)$ the corresponding $\Fund$-equivariant maps (recall \S\,\ref{map}).

\begin{lemma}\label{lox} For every proximal $\g\in\G,$ $\Fund\rho(\g)$ is proximal with attracting line $\mapa\xi(\g_+)$ and repelling hyperplane $\wk{\mapa}\xi(\g_-).$
\end{lemma}

\begin{proof}By passing to a finite cover we can assume that $\sf G$ is algebraically simply connected. In view of Proposition \ref{descent} we can also assume that $\sf G$ is simply laced and thus make use of Lusztig's canonical basis $\base_{\Fund_\sroot}$ (Proposition  \ref{>0pesos}).

By conjugating $\rho$ we may assume that $\xi(\g_+)=[\Borel]$ and that $\xi(\g_-)=[\wk{\Borel}].$  Since $\rho(\g)$ fixes both complete flags $\xi(\g_+)$ and $\xi(\g_-),$ it can be written as \begin{equation}\label{diagonalGamma}\rho(\g)=m_{\rho(\g)}\exp(a_\g)\end{equation} for a unique $a_\g\in\a$ and $m_{\rho(\g)}\in\sf M.$

The composition $\mapa\xi:\circulo\to\P(V)$ is a continuous $\Fund\rho$-equivariant map. By the assumptions $\xi(\g_+)=[\Borel]$ and $\xi(\g_-)=[\wk{\Borel}],$ one has $$\mapa\xi(\g_+)=V^+\textrm{ and }\wk\mapa\xi(\g_-)=V^-=\bigoplus_{\chi\in\poids(\fund)-\{\peso_\sroot\}}V^\chi$$ respectively.

By definition there exists $x\in\circulo$ distinct from $\g_+$ and $\g_-$  and $g\in \sf G_{>0}$ such that $\xi(x)=g\xi(\g_+).$ Lusztig's Proposition \ref{>0pesos} states, in particular, that if $v\in V^+$ is non-zero then $\Fund_\sroot(g)v=\sum_{\bf e\in\base_{\Fund}} c_{\bf e}\bf e$ with $c_{\bf e}>0$ for all $\bf e.$  

Additionally, Equation (\ref{diagonalGamma}) implies that $\Fund\rho(\g)$ is the commuting product of a matrix diagonal in $\base_{\Fund}$ and a finite order element. Let us denote thus by $ \Omega_{\bf e}(\g)$ the (possibly complex) eigenvalue of $\Fund\rho(\g)$ of the vector $\bf e\in\base_{\Fund}$ and let $\mu_1(\bar\phi\rho(\g))$ be the spectral radius of $\bar\phi\rho(\g)$.

If $k$ is the order of $m_{\rho(\g)},$ then $\Fund\big(\rho(\g)^k\big)$ is diagonal in $\base_{\Fund},$ so that $\Omega_{\bf e}(\g)^k\in\R$ and one has for all $n\in\N$\begin{equation}\label{diag}\frac1{\mu_1\big(\Fund\rho(\g)\big)^{nk}}\big(\Fund\rho(\g^{nk})\big)(gv)=\sum_{\bf e\in\base_{\Fund_\sroot}}\Big(\frac{\Omega_{\bf e}(\g)}{\mu_1\big(\Fund\rho(\g)\big)}\Big)^{nk}c_{\bf e}\bf e.\end{equation}

Since $\g^nx\to\g_+,$ equivariance implies $\Fund\rho(\g^n)(gV^+)\to V^+.$ Consequently, given that $c_{\bf e}>0,$ Equation \eqref{diag} yields $$|\Omega_{\bf e}(\g)|<\mu_1\big(\Fund\rho(\g)\big)$$ for every $\bf e$ except the one in $V^+$ and thus the spectral radius of $\Fund\rho(\g)$ is (only) attained on $V^+.$ Consequently $\Fund\rho(\g)$ is proximal and $\mapa\xi(\g_+)$ is the attracting point of $\bar\phi\rho(\g)$.\end{proof}

\begin{lemma}\label{parab}Let $\upupsilon\in\G$ be parabolic with fixed point $x_\upupsilon,$ then there exists $k\in\N_{\geq1}$ such that $\rho(\upupsilon^k)$ belongs to the unipotent radical of $\xi(x_\upupsilon),$ moreover, there exists an open set $\cal O\subset \cal F$ such that $\upupsilon^nz\to\xi(x_\upupsilon)$ for every $z\in\cal O.$\end{lemma}

\begin{proof} Again we can assume that $\sf G$ is simply laced and algebraically simply connected and make use of Lusztig's canonical basis $\base_{\Fund}$ (Proposition  \ref{>0pesos}). We assume moreover that $\xi(x_\upupsilon)=[\Borel]$ and that $[\wk\Borel]=\xi(z_0)$ for some auxiliary point $z_0\in\circulo.$ One has then $\mapa\xi(x_\upupsilon)=V^+.$ Let us write \begin{equation}\label{descomposicion}\rho(\upupsilon)=m_{\rho(\upupsilon)} \exp(a_\upupsilon) u_\upupsilon \end{equation} where $m_\upupsilon\in\sf M$ has finite order, commutes with $\exp a_\upupsilon\in A$ and normalizes $u_\upupsilon\in U.$ 

Since every element of $\bf e\in\base_{\Fund}$ belongs to a restricted weight space $V_{\chi_{\bf e}}$ of $\Fund,$ we can order $\base_{\Fund}$ so that $\bf e\geq\bf f$ if $\chi_{\bf e}\succ\chi_{\bf f},$ (the order between elements lying in the same weight space, or between weight spaces of the same level, is not relevant for the following). The elements of $A\cdot U$ are upper triangular in $\base_{\Fund},$ so if $k$ is the order of $m_{\rho(\upupsilon)}$ then the transformation $\Fund\rho(\upupsilon^k)$ is upper triangular in $\base_{\Fund}.$

Let us denote by $\mu_1=\exp\lambda_1\big(\Fund\rho(\upupsilon^k)\big)$ the spectral radius of $\Fund\rho(\upupsilon^k)$ and by $V_{\mu_1}$ the sum of Jordan blocks of $\Fund\rho(\upupsilon^k)$ associated to $\mu_1.$ Since $\Fund$ has values in $\SL(V)$ (because $\sf G$ is simple) one has $\mu_1\geq1$.

By Equation \eqref{descomposicion} and the definition of $\base_{\Fund}$, the intersection $V_{\mu_1}\cap\base_{\Fund}$ is a basis of $V_{\mu_1}.$ Denote by $\pi:V\to V_{\mu_1}$ the projection parallel to the vector space spanned by the remaining elements of $\base_{\Fund}.$  If $\ell\in \P(V)$ is not contained in $\ker\pi$ then one has \begin{equation}\label{con}d_\P\big(\Fund\rho(\upupsilon)^{kn}\cdot \ell,\P(V_{\mu_1})\big)\to0\end{equation} as $n\to\infty.$


By definition, there exists $x\in\circulo-\{x_\upupsilon,z_0\}$ and $g\in\sf G_{>0}$ such that $\xi(x)=g\cdot[\Borel].$ As before, if $v\in V^+$ is non-zero then $\Fund(g)v$ has positive coefficients in $\base_{\Fund}.$ This implies, in particular, that $\mapa\xi(x)=\Fund(g)V^+\nsubset\ker\pi.$ Since $\upupsilon^nx\to x_\upupsilon$ one has $\Fund\rho(\upupsilon)^n\big(\mapa\xi(x)\big)\to \mapa\xi(x_\upupsilon),$ which combined with Equation \eqref{con} gives $\mapa\xi(x_\upupsilon)\in V_{\mu_1}.$ In particular, $$\Fund(\rho(\upupsilon^k))v=\mu_1 v.$$ 


Additionally, since $\upupsilon^{-1}$ is also parabolic with fixed point $x_\upupsilon$, the above argument applied to $\upupsilon^{-1}$ gives that $v\in V^+$ belongs to the eigenspace associated to the spectral radius of $\Fund\rho(\upupsilon^{-k})$. However $\Fund\rho(\upupsilon^{-k})v=\mu_1^{-1}v$ so the spectral radius of $\Fund\rho(\upupsilon^{-k})$ is $\mu_1^{-1}\leq1$. Since the spectral radius of any element is at least 1, we obtain $\mu_1=1$ and that $\Fund\rho(\upupsilon)^k$ is upper triangular on $\base_{\Fund}$ with 1's in the diagonal, i.e. $\rho(\upupsilon^k)\in U.$ 

Considering $x\in\circulo-\{x_\upupsilon\}$ and $g\in\sf G_{>0}$ as before; one has that $\mapa\xi(x)=\Fund(g)V^+$ does not belong to a $\Fund\rho(\upupsilon)$-invariant subspace. Consequently, since $\Fund\rho(\upupsilon)^n\mapa\xi(x)\to \mapa\xi(x_\upupsilon),$ the same holds on a neighborhood of $\mapa\xi(x)$ and the lemma is proved.\end{proof}

The following is an immediate consequence of Proposition \ref{typeGamma}.

\begin{cor}\label{falta} If $\rho:\G\to\sf G$ is partially positive then it has minimal limit sets.
\end{cor}

\begin{proof}Indeed, Proposition \ref{typeGamma} readily implies that the limit set $\Lim{\rho(\G),\simple}=\xi(\circulo)$ and moreover that $\Lim{\rho(\G),\sroot}=p_\sroot\big(\xi(\circulo)\big),$ where $p_\sroot:\cal F\to\cal F_{\{\sroot\}}$ is the canonical projection. \end{proof}

\subsection{Proof of Theorem \ref{loOtro}} Corollary \ref{falta} together with Theorem \ref{tA} would complete the proof of Theorem \ref{loOtro} provided the Zariski closure of $\rho(\G)$ where reductive. The purpose of this subsection is thus to bypass the 'reductive Zariski closure' assumption. Consequently, Proposition \ref{semi} below and Theorem \ref{tA} prove Theorem \ref{loOtro}.

We begin by recalling the following lemma. It is a well known fact that the reader may check in Guéritaud-Guichard-Kassel-Wienhard \cite[\S 2.5.4]{GGKW} or in Benoist's lecture notes \cite{benoistnotas}.

\begin{lemma}\label{ZarRed} Let $\grupo$ be a group and let $\rho\in\hom(\grupo,\sf G)$ have non-solvable Zariski closure $\sf L.$ Let $\l=\h\oplus_\pi R_u(\h)$ be a Levi decomposition of the Lie algebra of $\sf L$ as a semi-direct product, with $\h$ reductive and $R_u(\h)$ its unipotent radical. Then there exists $\eta\in\hom(\grupo,\sf G)$ whose Zariski closure has Lie algebra $\h$ and a sequence $(g_n)\in\sf G$ with $g_n\rho g_n^{-1}\to\eta.$ 
\end{lemma}

As in Guéritaud-Guichard-Kassel-Wienhard \cite[\S 2.5.4]{GGKW}, we say that $\eta$ is the \emph{semi-simplification} of $\rho$ (regardless its Zariski closure is reductive and not necessarily semi-simple, and regardless of any uniqueness issues). We will also require the following slight modification of \cite[Proposition 4.13]{GGKW} whose proof works verbatim.

\begin{prop}[{\cite[Proposition 4.13]{GGKW}}]\label{semi-equi}Let $\grupo<\isom X$ be non-elementary and let $\rho:\Lambda\to\sf G$ be a representation with a continuous equivariant map $\xi:\bord X_\grupo\to\cal F$. Assume that $\xi$ is \begin{enumerate}\item transverse, i.e. for every $x\neq y$ one has $(\xi(x),\xi(y))\in\cal F^{(2)}$, and \item dynamics preserving, i.e. for every proximal $\g\in\grupo$ the image $\rho(\g)$ is proximal on $\cal F$ with attracting flag $\xi(\g_+)$.\end{enumerate} Then the semi-simplification $\eta:\grupo\to\sf G$ of $\rho$ also has a continuous $\eta$-equivariant map satisfying both these conditions.
\end{prop}

We proceed now to the main step.

\begin{prop}\label{semi}If $\rho:\G\to\sf G$ is partially positive then its semi-simplification $\eta$  has minimal limit sets and contains a positive loxodromic triple.
\end{prop}

\begin{proof}By continuity of the Jordan projection and Proposition \ref{typeGamma}, one has that $\eta(\g)$ is purely loxodromic for every proximal $\g\in\G,$ and that for every parabolic $\upupsilon\in\G$ there exists $k=k_\upupsilon$ such that $\eta(\upupsilon)^k$ is unipotent. 

Proposition \ref{semi-equi} gives a $\eta$-equivariant continuous transverse map $\xi_\eta:\circulo\to\cal F$ such that for every proximal $\g\in\G$ the flags $\xi_\eta(\g_+)$ and $\xi_\eta(\g_-)$ are respectively the attracting and repelling flags of $\eta(\g).$ The limit set $$\Lim{\eta(\G),\simple}=\xi_\eta\big(\circulo\big)$$ is thus minimal, and since every element of $\eta(\G)$ is either purely loxodromic, unipotent (up to a finite power) or elliptic, for every $\sroot\in\simple$ the limit set $\Lim{\eta(\G),\{\sroot\}}$ is the projection of $\Lim{\eta(\G),\simple\}}$ to $\cal F_{\{\sroot\}}$ and is thus also minimal.


In order to find a positive loxodromic triple in $\xi_\eta(\circulo),$ we observe that for every proximal $\g\in\G$ one has  $g_n\xi_\rho(\g_+)\to\xi_\eta(\g_+)$ as $n\to\infty.$ Indeed, for every $\sroot\in\simple$ the line $\mapa_\sroot\big(g_n\cdot\xi_\rho(\g)\big)$ is the eigenline of $\Fund_\sroot(g_n\rho(\g)g_n^{-1})$ associated to its spectral radius $$\lambda_1\big(\Fund_\sroot(g_n\rho(\g)g_n^{-1})\big)=\lambda_1\big(\Fund_\sroot(\rho(\g))\big)=\lambda_1\big(\Fund_\sroot(\eta(\g))\big).$$ Consequently, any accumulation point of $\big\{\mapa_\sroot\big(g_n\cdot\xi_\rho(\g)\big)\big\}$ is an eigenline associated to $\lambda_1\big(\Fund_\sroot(\eta(\g))\big);$ since $\Fund_\sroot\big(\eta(\g)\big)$ is proximal, this eigenline is $\mapa_\sroot\big(\xi_\eta(\g_+)\big).$

By assumption, there exists $x\in\circulo$ such that $\big(\xi_\rho(\g_+),\xi_\rho(x),\xi_\rho(\g_-)\big)$ is a positive triple of flags. By Proposition \ref{abierto}, $\posF$ is an open subset of $\cal F,$ thus, since attracting points of proximal elements are dense in $\circulo,$ there exists a proximal $h\in\G$ such that $\big(\xi_\rho(\g_+),\xi_\rho(h_+),\xi_\rho(\g_-)\big)$ is also a positive triple. 

We claim that $\big(\xi_\eta(\g_+),\xi_\eta(h_+),\xi_\eta(\g_-)\big)$ is a positive triple. Indeed, let us assume with out loss of generality that $\xi_\eta(\g_+)=[\Borel]$ and that $\xi_\eta(\g_-)=[\wk\Borel].$ One has the convergence $$g_n\cdot\big(\xi_\rho(\g_+),\xi_\rho(h_+),\xi_\rho(\g_-)\big)\to\big([\Borel],\xi_\eta(h_+),[\wk\Borel]\big)$$ and the triple $g_n\cdot\big(\xi_\rho(\g_+),\xi_\rho(h_+),\xi_\rho(\g_-)\big)$ is positive by definition. We may then also assume that for every $n,$ $g_n\cdot\xi_\rho(h_+)\in g_n\cdot \posF.$ The limit $\xi_\eta(h_+)$ of the sequence $g_n\cdot\xi_\rho(h_+)$ lies thus in the topological closure $\overline{\posF}.$ Proposition \ref{abierto} states that every element in the topological boundary of $\posF$ is not transverse to either $[\Borel]$ or $[\wk\Borel].$ However, as was observed earlier, $\xi_\eta(h_+)$ is both transverse to $[\Borel]$ and $[\wk\Borel]$ and thus necessarily lies in $\posF,$ the topological interior of $\overline{\posF}.$ \end{proof}

To wrap up the proof of Theorem \ref{loOtro}, we observe that the Zariski closure of $\rho(\G)$ and that of its semi-simplification $\eta(\G)$ have the same reductive part $\h$ (Lemma \ref{ZarRed}), and Proposition \ref{semi} permits to apply Theorem \ref{tA} to $\eta$, giving the desired conclusion.

\subsection{Hyperconvexity}To end this section we record the following remark that will be useful in Bridge\-man-Pozzetti-Wienhard-S. \cite{HessianHff}.

\begin{obs}Assume that $\circulo$ is homeomorphic to a circle, and that a partially positive $\rho:\G\to\sf G$ verifies the extra condition that $\xi$ sends positive ordered triples on $\circulo$ to positive triples of flags. Then for every $\sroot\in\simple$ and $x,y,z\in\circulo$ pairwise distinct one has $$\big(\mapa\xi(x)\oplus\mapa\xi(y)\big)\cap\wk\mapa_{\wedge^2\Fund_\sroot}\xi(z)=\{0\}.$$
\end{obs}

Here we interpret $\wk\mapa_{\wedge^2\Fund_\sroot}\xi(z)$ as a $\dim V_\sroot-2$-dimensional subspace of $V_\sroot.$ In the language of Pozzetti-S.-Wienhard \cite{PSW1}, the remark states that the curve $\mapa\xi(\circulo)$ is $(1,1,2)$-hyperconvex. 

\begin{proof}We can assume that $\sf G$ is simply laced and algebraically simply connected. We may also assume that $\xi(x)=[\wk\Borel],$ $\xi(z)=[\Borel]$ and that $\xi(y)=g\xi(x)$ for a $g\in \wk{U}_{>0}.$ We mimic now the proof of Corollary \ref{grupos->coherentes}. Since $\Fund_\sroot(g)$ has positive coefficients in the basis $\base_{\Fund_\sroot},$ the intersection of the plane $$P_y=\mapa\xi(x)\oplus\mapa\xi(y)=\mapa\xi(x)\oplus\Fund_\sroot(g)\mapa\xi(x)=V^+\oplus \Fund_\sroot(g)V^+$$ with $V^-,$ is not contained in any partial sum of restricted weight subspaces, in particular it is not contained in $$\sum_{\chi\in\poids(\fund_\sroot)-\{\peso_\sroot,\peso_\sroot-\sroot\}}V^\chi= \wk\mapa_{\wedge^2\Fund_\sroot}\xi(z)$$ as required.\end{proof}

\section{Group level}\label{gr}

Let us consider now a non-elementary discrete subgroup $\G<\isom(\gripe)$ of a proper Gromov-hyperbolic space $\gripe,$ a simple split $\sf G$ and the space $$\PP(\G,\sf G)=\{\rho:\G\to\sf G \textrm{ partially positive}\}.$$ In view of Proposition \ref{typeGamma}, if $\rho\in\PP(\G,\sf G)$ and $\g\in\G$ has infinite order, then the \emph{elliptic component} $m_{\rho(\g)}\in\sf M$ (as in Equations \eqref{diagonalGamma} or \eqref{descomposicion} according to the type of $\g$) is well defined. Once such $\g\in\G$ is fixed, we get a continuous map $\PP(\G,\sf G)\to\sf M$, $$\rho\mapsto m_{\rho(\g)},$$ and since $\sf M$ is finite, this map is locally constant. Its image is thus an invariant of the connected component of $\PP(\G,\sf G)$ containing $\rho.$

Let us consider $\rho\in\PP\big(\G,\SL_d(\R)\big)$ and denote by $\sf H$ the Zariski closure of $\rho(\G)$, we will use the above map to decide if $\rho(\G)$ is contained in $\sf H_0$, the identity component of $\sf H$.

\begin{lemma} The center of $\sf H$ is contained in $\{\pm\id\}$.
\end{lemma}

\begin{proof}Consider an element $z$ in the center of $\sf H$ and a proximal $\g\in\G$. The attracting line $\xi(\g_+)$ of $\rho(\g)$ is invariant by $z$, let $a\in\R$ be the eigenvalue of $z$ on $\xi(\g_+).$ The set $\{g\g_+:g\in\G\}$ is dense in $\bord X_\G$. Additionally, $g\g_+$ is the attracting line of $g\g g^{-1}$ and one sees that the eigenvalue of $z$ on $\xi(g\g_+)$ is also $a$. By Corollary \ref{reductiveClosure} $\sf H$ acts irreducibly on $\R^d$ so $\{\xi(g\g_+):g\in\G\}$ spans $\R^d$, giving that $z$ is a homothety.\end{proof}


Thus $\sf H_0$ is (conjugated to) one of the groups in Table \ref{lista} below. \begin{table}[h!]
  \begin{center}
    \begin{tabular}{l} 
   - $\SL_d(\R),$\\- a principal $\SL_2(\R),$\\ - $\Sp_{2n}(\R)$ if $d=2n$ for all $n\geq1,$\\ - $\SO_0(n,n+1)$ if $d=2n+1$ for all $n\geq1,$\\ - the fundamental representation for the short root of $\Ge_2$ if $d=7.$
    \end{tabular}
    \caption{Identity component of the Zariski closure of $\rho(\G).$} \label{lista}
  \end{center}
\end{table}

Observe that for every infinite order $\g\in\G$ the elliptic component $m_{\rho(\g)}\in\sf M\cap\sf H.$ This latter finite group is nothing but the centralizer in $\sf K_{\sf H}$ of $\exp\a_{\sf H},$ so if $m_{\rho(\g)}\in \sf H_0$ then $\rho(\g)\in\sf H_0.$

\begin{defi} 

A discrete and faithful morphism $\rho_0:\G\to\SL_d(\R)$ that factors as $$\G\to\SL_2(\R)\xrightarrow{\tau_d}\SL_d(\R),$$ where $\tau_d$ is a principal embedding, will be called \emph{principal}\footnote{This is usually referred to as \emph{Fuchsian} in the literature.}. Let us fix $\frak H$, a connected component of $\PP(\G,\SL_d(\R)\big)$ that contains a principal representation.\end{defi}

\begin{cor}\label{conexo} Assume $\G$ is torsion-free. Then for every $\rho\in\frak H$ the group $\rho(\G)$ is contained in the identity component of its Zariski closure.
\end{cor}

\begin{proof}Observe that the group  $$M:=\sf M_{\tau_d\big(\SL_2(\R)\big)}=\Big\{\tau\big(\begin{smallmatrix}-1 & 0 \\ 0 & -1\end{smallmatrix}\big),\tau\big(\begin{smallmatrix}1 & 0 \\ 0 & 1\end{smallmatrix}\big)\Big\}$$ is contained in all groups in Table \ref{lista}. If $\rho\in\frak H$ has Zariski closure $\sf H,$ then for every $\g\in\G-\{\id\}$ one has $m_{\rho(\g)}\in M\subset \sf H_0$ and thus $\rho(\g)\subset\sf H_0$. \end{proof}

Finally, let $S$ be a closed connected orientable surface of genus $\geq2$ and let $\rho:\pi_1S\to\PSL_d(\R)$ belong to a Hitchin component. Assume first that $\rho$ lifts to a representation $\widetilde\rho:\pi_1S\to\SL_d(\R).$ Then Theorem \ref{HP} assures that $\widetilde\rho\in\frak H(\pi_1S,\SL_d(\R))$ and Corollary \ref{conexo} implies that the Zariski closure of $\rho$ is the projectivisation of a group in Table \ref{lista}.  The  following lemma completes thus the proof of  Guichard's classification (Corollary \ref{gui}).

\begin{lemma} Every Hitchin representation $\rho:\pi_1S\to\PSL_d(\R)$ lifts to a representation with values in $\SL_d(\R)$.
\end{lemma}

\begin{proof}Culler's Theorem 4.1 in \cite{Culler} implies that every Hitchin representation lifts provided one of them does. Additionally, if $\eta:\pi_1S\to\PSL_2(\R)$ is discrete and faithful, then $\eta(\pi_1S)$ is in particular torsion-free so \cite[Corollary 2.3]{Culler} implies that $\eta$ lifts to a representation in $\SL_2(\R)$, giving the desired lemma.\end{proof}



\begin{obs} \item \begin{itemize}\item The case of Hitchin representations with values in $\SO_0(n,n)$ has been treated by Carvajales-Dai-Pozzetti-Wienhard \cite[Corollary 7.10]{CDPW}. \item The above argument for Hitchin representations in $\SL_d(\R)$ also applies to the cusped Hitchin representations studied by Canary-Zhang-Zimmer \cite{CuspedCZZ}.\end{itemize}
\end{obs}


\appendix

\section{The Hasse diagrams for extremal roots}\label{figurasdiagramas}

In this appendix we prove Lemma \ref{diagramasClasif}. To this end we compute the Hasse diagrams for the extremal roots of irreducible reduced root systems and compute, in a case by case manner, the existence/non-existence of surjective level preserving maps between them. Let us simplify notation and denote, for a simple root $x\in\simple_{\sf J}$ of some root system $\sf J,$ by $\hasse{\sf J}x$ the Hasse diagram $\hasse{\sf J}{\peso_x}$ for the fundamental weight $\peso_x.$

Most of the situations are ruled out by the following simple facts. If $f:\simple_{\sf L}\to\simple_{\sf J}$ is surjective and $\T^f:\hasse{\sf L}{\aa}\to\hasse{\sf J}{f(\aa)}$ is a surjective diagram map with labeling $f$ then:\begin{itemize}\item[-]$\rk\sf J\leq\rk\sf L,$\item[-]both $\hasse{\sf L}{\aa}$ and $\hasse{\sf J}{f(\aa)}$ have the same total amount of levels,\item[-] if $\chi$ is the only vertex at a given level, then the number of arrows pointing downwards in $\hasse{\sf L}{\aa}$ is greater than that of $\T^f(\chi)$ in $\hasse{\sf J}{f(\aa)},$ \item[-] to show non-existence of such $f,$ it it sufficient to find one extremal root of $\sf L$ whose Hasse diagram does not surject to any diagram of $\sf J$ (for extremal roots).\end{itemize}

We refer the reader to the corresponding figures for the labeling of simple roots for each Dynkin diagram.

\begin{lemma}Leaving aside the case $f=$identity, one has the following.\begin{itemize}\item[- Type $\A:$] The only surjective diagram map $\T^f:\hasse{\A_d}{\bb_1}\to\hasse{\sf J}{x}$ with $x$ extremal are \begin{itemize}\item[-]$d=2n$ and $\sf J=\B_n$ and $x=\bb$ for all $n$ and moreover $\Ge_2$ and $x=\aa$ if $d=6,$\item[-] $d=2n-1,$ $\sf J=\Ce_n$ and $x=\bb.$ \end{itemize}\item[- Type $\B:$]The only surjective diagram maps $\T^f:\hasse{\B_n}{\bb}\to\hasse{\sf J}{x}$ with $x$ extremal is $n=3$ and $\sf J=\Ge_2$ and $x=\aa.$\item[- Type $\Ce:$]There is no surjective diagram map $\T^f:\hasse{\Ce_n}{\bb}\to\hasse{\sf J}{x}$ with $x$ extremal.\item[- Type $\D:$] The only surjective diagram maps $\T^f:\hasse{\D_n}{\bb}\to\hasse{\sf J}{x}$ with $x$ extremal are \begin{itemize}\item[-] $\sf J=\B_{n-1}$ with $x=\bb$ for all $n,$\item[-] moreover one has $\sf J=\B_3$ with $x=\aa$ and $\sf J=\Ge_2$ with $x=\aa$ if $n=4.$\end{itemize}\end{itemize}
\end{lemma}

\begin{proof}Observe that all Hasse diagrams $\hasse{\A_n}{\bb_1}$, $\hasse{\B_n}{\bb}$ (Figure (\ref{HasseAn})) and $\hasse{\Ce_n}{\bb}$ (Figure (\ref{HasseCn})) consist on exactly one arrow exiting each vertex. By restricting the total amount of levels given by the existence of $\T^f$ together with the fact that $\rk\sf J\leq n$ (in each case) one completes the proof. A similar argument works for $\hasse{\D_n}{\bb_1}$ (see also Figure (\ref{d4b3})). \end{proof}

We now treat the type $\EE$ family, we will show that there is no surjective diagram map from $\hasse{\EE_k}{\aa}$ for $k=6,$ $7$ or $8$ to any other Hasse diagram $\hasse{\sf J}{x}$ with extremal $x$, except for $\hasse{\EE_6}{\aa}\to\hasse{\F_4}{\aa}$ (as shown in Figure (\ref{E6-F4})).

\begin{lemma} There is no surjective map $\T^f$ from $\hasse{\EE_k}{\aa}$ for $k=6,$ $7$ or $8$ onto any of $\hasse{\A_n}{\bb}\approx\hasse{\A_n}{\aa}, $ $\hasse{\B_n}{\bb},$ $\hasse{\Ce_n}{\bb},$ for $n\leq8$ nor onto $\hasse{\Ge_2}{\bb}$ or $\hasse{\Ge_2}{\aa}.$
\end{lemma}

\begin{proof} The non-existence of such map comes from the fact that $\hasse{\EE}{\aa}$ has too many levels (compared to the fact that $n$ must be smaller than $k$), observe that Figure (\ref{HasseE}) depicts $\hasse{\EE_k}{\aa}$ up to levels $9,$ $10$ and $11$ respectively for $k=6,$ $7$ or $8.$ The case $\hasse{\Ge_2}{\aa}$ is readily discarded since it has 7 levels.

We now treat $\hasse{\sf J}{x}$ for $\sf J=\A_n,$ $\B_n,$ $\Ce_n$ and $x=\bb.$ Since these diagrams consist on only one arrow pointing downwards at each level, from Figure (\ref{HasseE}) one sees that if such a $\T^f$ existed then necessarily $$f(\bb_2)=f(\sroot)=f(\bb)=f(\bb_3)=f(\bb_4).$$ Since $f$ is surjective, the above equalities imply that $\sf J$ has rank $\leq k-4,$  that is $n\leq k-4\leq4.$ However $\hasse{\A_4}{\bb}$ has 5 levels, $\hasse{\B_4}{\bb}$ has 9 levels and $\hasse{\Ce_4}{\bb}$ has $8$ levels, but $\hasse{\EE}{\aa}$ has at least 9 levels (actually at least 17 	as seen in Figure (\ref{E6-F4})).

Finally, from Figure (\ref{HasseG}) one sees that $\hasse{\Ge_2}{\bb}$ has 14 levels but Figure (\ref{E6-F4}) shows that $\hasse{\EE}{\aa}$ has at least $17$ levels.\end{proof}

\begin{lemma} \item\begin{itemize}\item[-] There is no surjective map $\T^f$ from $\hasse{\EE_k}{\aa}$ $k=6,$ $7$ or $8$ onto $\hasse{\B_n}{\aa},$ $\hasse{\Ce_n}{\aa},$ $\hasse{\D_n}{\aa}\approx\hasse{\D_n}{\sroot},$ $\hasse{\EE_j}{\sroot}$ ($j=6,$ $7$ or $8$), $\hasse{\EE_{k-1}}{\aa},$ (if $k=7$ or $8$) $\hasse{\EE_{k-2}}{\aa}$ (if $k=8$).
\item[-] There is no surjective map $\T^f$ from $\hasse{\EE_7}{\aa}$ or $\hasse{\EE_8}{\aa}$ onto $\hasse{\EE_6}{\aa},$ $\hasse{\EE_j}{\bb}$ ($j=6,$ $7$ or $8$), $\hasse{\F_4}{\bb}$ and $\hasse{\F_4}{\aa}.$\end{itemize}\end{lemma}

\begin{proof} In $\hasse{\EE_k}{\aa}$ the first level with more than one exiting arrow is at least 4, however the diagrams appearing in the first item have $2$ exiting arrows at the third level. Similarly the first level with more than one exiting arrow in $\hasse{\EE_7}{\aa}$ or $\hasse{\EE_8}{\aa}$ is at least $5,$ but the diagrams listed in the second item have earlier multiple exiting arrows.\end{proof}

The $\EE$ family is thus achieved with the next Lemma.

\begin{lemma}There is no surjective map $\T^f$ from $\hasse{\EE_k}{\aa}$ for $k\in\{6,7,8\}$ onto $\hasse{\D_n}{\bb}.$
\end{lemma}

\begin{proof}Since in $\hasse{\D_n}{\bb}$ there is only one arrow starting at each node for every level up to $n-2,$ if such a $\T^f$ exists then one must have $n-2=k-3.$ However, by looking at the levels after the first rhombus in Figure (\ref{HasseE}) one sees that $$f(\bb)=f(\sroot)=(\bb_3),$$ thus $n\leq k-2,$ which is a contradiction with $n=k-1.$\end{proof}

We now deal with  $\F_4$ and $\Ge_2$.

\begin{lemma}\item \begin{itemize}\item[-]Let $x$ be an extremal root of $\F_4$. Then, other than $f=\id$, there is no surjective map $\T^f$ from $\hasse{\F_4}{\aa}$ to any other Hasse diagram $\hasse{\sf J}{z}$ for extremal $z$.\item[-]Let $x$ be a root of $\G_2$. There, other than $f=\id$, there is no surjective map $\T^f$ from $\hasse{\Ge_4}{x}$ to any other Hasse diagram $\hasse{\sf J}{z}$ for extremal $z$.
\end{itemize}
\end{lemma}

\begin{proof} Follows easily since the other reduced root systems with $\rk \sf J\leq4$ and $\leq2$ respectively do not have enough levels.
\end{proof}

\begin{figure}
\begin{tikzpicture}[baseline= (a).base]

\node[scale=.8] (a) at (-3,0){
\begin{tikzcd}[column sep=small]
 & \circ \arrow[d, "\bb_1"]& & \\ 
  & \circ \arrow[d, "\bb_2"]& & \\
 & \vdots \arrow[d, "\bb_{n-1}"]& & \\
 & \circ\arrow[d, "\bb_n"]&  & \\
 & \circ  & &
\end{tikzcd}};
\node[scale=.8] at (-3,-4){$\hasse{\A_n}{\peso_{\bb_1}}$ };

\node[scale=.8] (a) at (-1.5,0){
\begin{tikzcd}[column sep=small]
 & \circ \arrow[d, "\bb_n"]& & \\ 
  & \circ \arrow[d, "\bb_{n-1}"]& & \\
 & \vdots\arrow[d, "\bb_{2}"]& & \\
 & \circ\arrow[d, "\bb_1"]&  & \\
 & \circ  & &
\end{tikzcd}};
\node[scale=.8] at (-1.5,-4){$\hasse{\A_n}{\peso_{\bb_n}}$ };

\node[scale=.8] at (-2.5,-5){\dynkin[labels={\bb_1,\bb_2,,\bb_n},scale=1.4] A{}};

\node[scale=.8] (a) at (0,0){
\begin{tikzcd}[column sep=small]
 & \circ \arrow[d, "\bb"]& & \\ 
  & \circ \arrow[d, "\bb_2"]& & \\
 & \vdots\arrow[d, "\bb_{n-1}"]&  & \\
 & \circ \arrow[d, "\aa"]& & \\ 
  & \circ \arrow[d, "\aa"]& & \\
   & \circ\arrow[d, "\bb_{n-1}"]&  & \\
 & \vdots\arrow[d, "\bb"]&  & \\
 & \circ  & &
\end{tikzcd}};
\node[scale=.8] at (0,-4){$\hasse{\B_n}{\peso_{\bb}}$};

\node[scale=.8] at (4,0){\begin{tikzcd}[column sep=small,]
 && \circ \arrow[d,   "\aa"]& && &&\\ 
 && \circ \arrow[d,   "\bb_{n-1}"]&  &&&& \\
 && \circ \arrow[dl,  swap, "\aa"] \arrow[dr,  "\bb_{n-2}"]  & &&&&\\
&  \circ  \arrow[dr,  "\bb_{n-2}"] &  & \arrow[dl,  "\aa"] \circ  \arrow[dr,  "\bb_{n-3}"] &&& &\\
 && \circ \arrow[dl,  swap, "\bb_{n-1}"]  \arrow[dr,  "\bb_{n-3}"] &&\circ \arrow[dl,  "\aa"]\arrow[dr,  "\bb_{n-4}"] &&&\\
 &\circ\arrow[d,swap,"\aa"]\arrow[dr, "\bb_{n-3}"]&  & \circ\arrow[dl, "\bb_{n-1}"]\arrow[dr, "\bb_{n-4}"]&    &\circ\arrow[dl,"\aa"]\arrow[dr,"\bb_{n-5}"]&&\\
&\circ &\circ\arrow[d,swap,"\aa" description] \arrow[dr, "\bb_{n-4}"]\arrow[dl, near end, swap, "\bb_{n-2}"]& & \circ\arrow[dl, "\bb_{n-1}"]\arrow[dr, "\bb_{n-5}"]&   &\circ\arrow[dl,"\aa"]\arrow[dr,"\bb_{n-6}"]\\
&\circ& \circ\arrow[ul,near end, crossing over,"\bb_{n-3}" description]&\circ & &\circ&&\circ
\end{tikzcd}};
\node[scale=.8] at (4,-4){$\hasse{\B_n}{\peso_{\aa}}$ up to level 8};
\node[scale=.8] at (2,-5){\dynkin[labels={\bb,\bb_2,,\bb_{n-1},\aa},scale=1.4] B{}};

%
%
%
%
%

\end{tikzpicture}\caption{Hasse for extremal roots of $\A_n$ (left) and $\B_n$ (right)}\label{HasseBn}\label{HasseAn}
\end{figure}

%
%
%
%
%

\begin{figure}
\begin{tikzpicture}[baseline= (a).base]

\node[scale=.8] (a) at (4,0.6){
\begin{tikzcd}[column sep=small]
 & \circ \arrow[d, dash, "\bb"]& & \\ 
 & \circ\arrow[d, dash, "\bb_2"]&  & \\
 &\vdots\arrow[d,dash,"\bb_{n-2}"]&&\\
 & \circ \arrow[dl, dash,swap,"\sroot"] \arrow[dr, dash,"\aa"]  & &\\
  \circ  \arrow[dr, dash,"\aa"] &  & \arrow[dl, dash,"\sroot"] \circ  & \\
  & \circ  \arrow[d, dash,"\bb_{n-2}"]&   &   \\
  & \vdots\arrow[d, dash,"\bb"]& & \\
 & \circ &  &
\end{tikzcd}};
\node[scale=.8] at (4,-3.5){$\hasse{\D_n}{\peso_{\bb}}$};

\node[scale=.8] at (7,0.6){
\begin{tikzcd}[column sep=small]
 & \circ \arrow[d, dash,"\aa"]& & \\ 
  & \circ \arrow[d, dash,"\bb_{n-2}"]& & \\
 & \circ\arrow[d, dash,swap, "\bb_{n-3}"]\arrow[dr,dash,"\sroot"]& & \\
 & \circ\arrow[d, dash,swap, "\bb_{n-4}"]\arrow[dr,dash,"\sroot"]&\circ \arrow[d,dash,"\bb_{n-3}"] & \\ 
 & \circ\arrow[d, dash,swap, "\bb_{n-5}"]\arrow[dr,dash,"\sroot"]&\circ \arrow[d,dash,"\bb_{n-4}"] & \\ 
 & \circ\arrow[d, no head,dashed]&  \circ\arrow[d,no head, dashed]& \\
 & \circ\arrow[dr,dash,"\sroot"]&\circ \arrow[d,dash,"\bb"] & \\
 & &  \circ&
\end{tikzcd}};
\node[scale=.8] at (7,-3.5){$\hasse{\D_n}{\peso_{\aa}}\approx\hasse{\D_n}{\peso_{\sroot}}$};

\node[scale=.8] at (5,-4.5){\dynkin[labels={\bb,\bb_2,,,\sroot,\aa},scale=1.4] D{}};

\node[scale=.8] (a) at (-3,0){
\begin{tikzcd}[column sep=small]
  & \circ \arrow[d, "\bb"]& & \\ 
  & \circ \arrow[d, "\bb_2"]& & \\
 & \vdots\arrow[d, "\bb_{n-1}"]&  & \\
 & \circ \arrow[d, "\aa"]& & \\ 
  & \circ \arrow[d, "\bb_{n-1}"]& & \\
 & \vdots\arrow[d, "\bb"]&  & \\
 & \circ  & &
\end{tikzcd}};
\node[scale=.8] at (-3,-3.5){$\hasse{\Ce_n}{\peso_{\bb}}$};

\node[scale=.8] at (0,0){\begin{tikzcd}[column sep=small,]
& & \circ \arrow[d,   "\aa"]& & \\ 
& & \circ \arrow[d,   "\bb_{n-1}"]&  & \\
& & \circ \arrow[dl,  swap, "\bb_{n-1}"] \arrow[dr,   "\bb_{n-2}"]  & &\\
&  \circ  \arrow[dr,  "\bb_{n-2}"]\arrow[dl, swap,"\aa"] &  & \arrow[dl,  "\bb_{n-1}"] \circ  \arrow[dr,  "\bb_{n-3}"] & \\
\circ \arrow[dr, "\bb_{n-2}"] & &\circ\arrow[dl, "\aa"]\arrow[dr, "\bb_{n-3}"]  &  & \circ\arrow[dl,"\bb_{n-1}"]   &\\
&\circ   & & \circ    &  &\end{tikzcd}};
\node[scale=.8] at (0,-3.5){$\hasse{\Ce_n}{\peso_{\aa}}$ up to level 6};

\node[scale=.8] at (-1,-4.5){\dynkin[labels={\bb,\bb_2,,\bb_{n-1},\aa}, scale=1.4] C{}};

%

\end{tikzpicture}\caption{Hasse for extremal roots of $\Ce_n$ (left) and $\D_n$ (right)}\label{HasseDn}\label{HasseCn}
\end{figure}

\begin{figure}
\begin{tikzpicture}
\node[scale=.8]   at (-6,0){
\begin{tikzcd}[column sep=small,]
 &&& \circ \arrow[d,   "\bb"]& &&&& \\ 
 &&& \circ \arrow[d,   "\bb_2"]& & &&&\\ 
 &&& \circ \arrow[d,   "\bb_3"]&  &&& &\\
 &&& \circ \arrow[dl,   swap]{}{\sroot} \arrow[dr, dash]{}{\bb_4}  & &&&&\\
&&\circ  \arrow[dr,   swap,  "\bb_4"] &  & \arrow[dl,   swap, "\sroot"] \circ  \arrow[dr,  "\bb_5"] & &&&\\
&& & \circ \arrow[dr,  "\bb_5"] \arrow[dl,  swap, "\bb_3"] &  & \circ \arrow[dr,  "\bb_6"]\arrow[dl,   swap, "\sroot"] &&&\\
&&\circ \arrow[dl, swap, "\bb_2"]\arrow[dr,  "\bb_5"]&& \circ\arrow[dl,  "\bb_3"]\arrow[dr,  "\bb_6"] &&\circ\arrow[dl,  "\sroot"]\arrow[dr,  "\aa"]&&\\
&\circ\arrow[dl,  swap,"\bb"]\arrow[dr,  "\bb_5"]&&\circ\arrow[dl,  swap,"\bb_2"]\arrow[d,  "\bb_4"]\arrow[dr,  "\bb_6"]&&\circ\arrow[dl,  "\bb_3"]\arrow[dr,  "\aa"]&&\circ\arrow[dl,  "\sroot"]&\\
\circ\arrow[dr, "\bb_5"]&& \circ\arrow[dl, swap,"\bb"]\arrow[d,  "\bb_4"]\arrow[dr,  "\bb_6"]&\circ\arrow[dl,  crossing over,"\bb_2"]& \circ\arrow[dl,  "\bb_2"]\arrow[d,  "\bb_4"]\arrow[dr,  "\aa"]&&\circ\arrow[dl,  "\bb_3"]&&\\
&\circ& \circ&\circ\arrow[d, dashed]& \circ\arrow[ul, crossing over, "\bb_6"]&\circ&&&\\
& & & {}  &   &    &  & &
\end{tikzcd}};
\node[scale=.8] at (-6,-4.5){$\hasse{\EE}{\peso_{\bb}}$ up to level 10};

\node[scale=.8] at (0,0){
\begin{tikzcd}[column sep=small,]
& & \circ \arrow[d, dash,"\sroot"]& & &&\\ 
& & \circ \arrow[d, dash,"\bb_3"]& & &&\\
& & \circ\arrow[dl, dash,swap, "\bb_{2}"]\arrow[dr,dash,"\bb_4"]& & &&\\
 &   \circ\arrow[dl, dash,swap, "\bb"]\arrow[dr,dash,"\bb_4"] & &\circ\arrow[dl, dash, "\bb_2"]\arrow[dr,dash,"\bb_5"]& &&\\
\circ\arrow[dr,dash,"\bb_4"] & &  \circ\arrow[dl,dash,swap,"\bb"]\arrow[d,dash,"\bb_3"]\arrow[dr,dash,"\bb_5"] & & \circ\arrow[dl,dash,"\bb_2"]\arrow[dr,dash,"\bb_6"]&&\\
&\circ\arrow[dr,dash,near end,swap,"\bb_5"]\arrow[d,dash,near start,swap, "\bb_3"] &\circ\arrow[dl,dash,crossing over, near start, "{\scriptscriptstyle\bb}" description] &\circ\arrow[dl,dash,near end, "{\scriptscriptstyle\bb}"]\arrow[d,dash,"\bb_3"]\arrow[dr,dash,"\bb_6"]& &\circ\arrow[dl, dash, "\bb_2"]\arrow[dr, dash, "\aa"]&\\
&\circ &\circ &\arrow[d,dashed]\arrow[ul,dash,near end,swap, crossing over, "{\scriptscriptstyle \bb_5}" description]\circ&\circ& &\circ\\
&& &{}&&&
\end{tikzcd}
};\node[scale=.8] at (0,-4.5){$\hasse{\EE}{\peso_{\sroot}}$ up to level 7};
%

\node[scale=1] at (-3,-5){\dynkin[labels={\bb,\sroot,\bb_2,,,,\bb_6,\aa},scale=1.4] E8};

\end{tikzpicture}\caption{Hasse for extremal roots of the $\EE$ family}\label{HasseE}
\end{figure}

\begin{figure}
\begin{tikzpicture}
\node[scale=.8] at (0,0.6){
\begin{tikzcd}[column sep=small]
 && \circ \arrow[d,  "\aa"]& && \\ 
 & & \circ \arrow[d,   "\bb_6"]& & &\\
& & \vdots\arrow[d,  "\bb_3"]&  && \\
 & & \circ\arrow[dl,  swap, "\bb_{2}"]\arrow[dr, "\sroot"]& && \\
 &   \circ\arrow[dl,  swap, "\bb"]\arrow[dr, "\sroot"] & &\circ\arrow[dl,   "\bb_2"] &&\\
 \circ\arrow[dr,  swap, "\sroot"]& &\circ \arrow[dl,  swap,"\bb"]\arrow[dr,  swap,"\bb_3"]& && \\
& \circ\arrow[dr,  swap, "\bb_3"]& &\circ\arrow[dl,  swap,"\bb"]\arrow[dr,  swap,"\bb_4"]&& \\
& & \circ\arrow[dr,  swap, "\bb_4"]\arrow[dl,  swap, "\bb_2"]& &\circ\arrow[dl,  swap,"\bb"]\arrow[dr,  swap,"\bb_5"] &\\
& \circ& &\arrow[d,dashed]\circ&&\circ\\
& & &{}& &
\end{tikzcd}};\node[scale=1] at (0,-4){first levels of $\hasse{\EE}{\peso_{\aa}}$};

\end{tikzpicture}\caption{Hasse for extremal roots of the $\EE$ family, continued}
\end{figure}

\begin{figure}
\begin{tikzpicture}
\node[scale=.8]  (a) at (-5,0){
\begin{tikzcd}[column sep=small,]
&&  & \circ \arrow[d,   "\bb"]& & \\ 
&& & \circ \arrow[d,   "\nu"]& & \\ 
&& & \circ \arrow[d,   "\sroot"]&  & \\
&& & \circ \arrow[dl,   swap, "\sroot"] \arrow[dr,  "\aa"]  & &\\
&&  \circ  \arrow[dr,   swap,  "\aa"]\arrow[dl,   swap,  "\nu"] &  & \circ \arrow[dl,   "\sroot"]  & \\
& \circ\arrow[dl, "\bb"] \arrow[dr, "\aa"]& & \circ \arrow[dr,  "\aa"] \arrow[dl,  swap, "\nu"] &  &\\
\circ\arrow[dr, "\aa"] &  & \circ \arrow[dl,  swap, "\bb"]\arrow[d, "\sroot"] \arrow[dr,   "\aa"]&   &  \circ \arrow[dl,   "\nu"]    & \\
 &  \circ\arrow[d, swap,"\sroot"] \arrow[dr, near start, "\aa" description] &  \circ\arrow[dl, near end, crossing over,"\bb" description]  & \circ\arrow[d, "\sroot"]\arrow[dl, near end,swap,"\bb" description]     &  &\\
 &  \circ\arrow[dl, "\nu"]\arrow[dr, "\aa"]  &  \circ\arrow[d, "\sroot"]  & \circ\arrow[ul, crossing over,swap, near end, "\aa" description]\arrow[dl, "\bb"] \arrow[dr, "\sroot"]    &  &\\
\circ \arrow[dr, "\aa"] \arrow[d, "\sroot"]    &  &  \circ  \arrow[dl, "\nu"] \arrow[dr, "\sroot"]   &      &  \circ \arrow[dl, "\bb"]  &\\
\circ \arrow[drr, swap,"\aa"] &  \circ \arrow[dr, "\sroot"] &    & \circ   \arrow[dl, "\nu"]  &   &\\
 &   &  \circ\arrow[d, dashed]  &  &  &\\
 &   &  {}  &    &  &
\end{tikzcd}};
\node[scale=.8] at (-5,-7.2){$\hasse{\F_4}{\peso_{\bb}}$ up to level 12};

\node[scale=.8] at (-1,0.6){
\begin{tikzcd}[column sep=small]
 && \circ \arrow[d,  "\aa"]& && \\ 
 & & \circ \arrow[d,   "\eps"]& & &\\
& & \circ\arrow[d,  "\nu"]&  && \\
 & & \circ\arrow[dl,  swap, "\eps"]\arrow[dr, "\sroot"]& && \\
 &   \circ\arrow[dl,  swap, "\aa"]\arrow[dr, "\sroot"] & &\circ\arrow[dl,   "\eps"] &&\\
 \circ\arrow[dr,  swap, "\sroot"]& &\circ \arrow[dl,  swap,"\aa"]\arrow[dr,  swap,"\nu"]& && \\
& \circ\arrow[dr,  swap, "\nu"]& &\circ\arrow[dl,  swap,"\aa"]\arrow[dr,  swap,"\eps"]&& \\
& & \circ\arrow[dr,  swap, "\eps"]& &\circ\arrow[dl,  swap,"\aa"] &\\
& & &\circ&&\\
& & \circ\arrow[ur, "\eps"]& &\circ\arrow[ul,"\aa"] &\\
& \circ\arrow[ur, "\nu"]& &\circ\arrow[ul,"\aa"]\arrow[ur,"\eps"]&& \\
 \circ\arrow[ur, "\sroot"]& &\circ \arrow[ul,"\aa"]\arrow[ur,"\nu"]& && \\
 &   \circ\arrow[ul,  "\aa"]\arrow[ur, "\sroot"] & &\circ\arrow[ul,   "\eps"] &&\\
 & & \circ\arrow[ul,"\eps"]\arrow[ur, "\sroot"]& && \\
& & \circ\arrow[u,  "\nu"]&  && \\
& & \circ \arrow[u,   "\eps"]& & &\\
 && \circ \arrow[u,  "\aa"]& && 
\end{tikzcd}};
\node[scale=.8] at (-1,-7.2){$\hasse{\F_4}{\peso_{\aa}},$ it has 17 levels. };

\node[scale=.8] at (-3,-8){\dynkin[labels={\bb,\nu,\sroot,\aa},scale=1.4] F4};

\node[scale=.8]  at (2.5,0){
\begin{tikzcd}[column sep=small,]
 & \circ \arrow[d, "\bb"]&  \\ 
 & \circ \arrow[d, dash, "\aa"]&  \\ 
 & \circ \arrow[d, dash, "\aa"]&   \\
 & \circ \arrow[dl, dash, swap, "\bb"] \arrow[dr, dash,"\aa"]  & \\
  \circ  \arrow[dr, dash, swap,  "\aa"] &  & \circ \arrow[dl, dash, "\bb"]   \\
  & \circ \arrow[dr,dash, "\aa"] \arrow[dl,dash, swap, "\bb"] &   \\
 \circ \arrow[dr, dash, swap, "\aa"] & & \circ \arrow[dl, dash,"\bb"]   \\
  & \circ \arrow[d,dash,"\aa"]  &       \\
  & \circ \arrow[d,dash,"\aa"]  &       \\
  & \circ \arrow[d,dash,"\bb"]  &       \\
  & \circ&
\end{tikzcd}};
\node[scale=.8] at (2.5,-7.2){$\hasse{\Ge_2}{\peso_{\bb}}$};
\node[scale=.8] at (4.5,0){
\begin{tikzcd}[column sep=small]
 & \circ \arrow[d, dash]& & \\ 
  & \circ \arrow[d, dash]& & \\
 & \circ\arrow[d, dash]&  & \\
 & \circ \arrow[d, dash]& & \\
 & \circ \arrow[d, dash]& & \\
 & \circ \arrow[d, dash]& & \\
 & \circ  & &
\end{tikzcd}};\node[scale=.8] at (4.5,-7.2){$\hasse{\Ge_2}{\peso_{\aa}}$};

\node[scale=.8] at (3.5,-8){\dynkin[labels={\bb,\aa},scale=1.4] G2};

\end{tikzpicture}\caption{Hasse for extremal roots of $\F_4$ (left) and $\Ge_2$ (right)}\label{HasseF}\label{HasseG}
\end{figure}

\begin{figure}[h]
\begin{tikzpicture}
\node[scale=1] at (-4,0.6){
\begin{tikzcd}[column sep=small]
 && \circ \arrow[d,  "\aa"]& && \\ 
 & & \circ \arrow[d,   "\bb_4"]& & &\\
& & \circ\arrow[d,  "\bb_3"]&  && \\
 & & \circ\arrow[dl,  swap, "\bb_{2}"]\arrow[dr, "\sroot"]& && \\
 &   \circ\arrow[dl,  swap, "\bb"]\arrow[dr, "\sroot"] & &\circ\arrow[dl,   "\bb_2"] &&\\
 \circ\arrow[dr,  swap, "\sroot"]& &\circ \arrow[dl,  swap,"\bb"]\arrow[dr,  swap,"\bb_3"]& && \\
& \circ\arrow[dr,  swap, "\bb_3"]& &\circ\arrow[dl,  swap,"\bb"]\arrow[dr,  swap,"\bb_4"]&& \\
& & \circ\arrow[dr,  swap, "\bb_4"]\arrow[dl,  swap, "\bb_2"]& &\circ\arrow[dl,  swap,"\bb"]\arrow[dr,  swap,"\aa"] &\\
& \circ& &\circ&&\circ\\
& & \circ\arrow[ur, "\bb_2"]\arrow[ul, "\bb_4"]& &\circ\arrow[ul,"\aa"]\arrow[ur,"\bb"] &\\
& \circ\arrow[ur, "\bb_3"]& &\circ\arrow[ul,"\aa"]\arrow[ur,"\bb_2"]&& \\
 \circ\arrow[ur, "\sroot"]& &\circ \arrow[ul,"\aa"]\arrow[ur,"\bb_3"]& && \\
 &   \circ\arrow[ul,  "\aa"]\arrow[ur, "\sroot"] & &\circ\arrow[ul,   "\bb_4"] &&\\
 & & \circ\arrow[ul,"\bb_{4}"]\arrow[ur, "\sroot"]& && \\
& & \circ\arrow[u,  "\bb_3"]&  && \\
& & \circ \arrow[u,   "\bb_4"]& & &\\
 && \circ \arrow[u,  "\aa"]& && 
\end{tikzcd}};

\node[scale=1] at (-4.5,-9){$\hasse{\EE_6}\aa\approx\hasse{\EE_6}\bb$};
\node[scale=1] at (-4.5,-10){\dynkin[labels={\bb,\sroot,\bb_2,\bb_3,\bb_4,\aa},scale=1.4] E6};

\node[scale=1] (a) at (0,1){$\begin{dynkinDiagram}E6\draw[thick] (root 1) to [out=-45, in=-135] (root 6);\draw[thick] (root 3) to [out=-45, in=-135] (root 5);\end{dynkinDiagram}$};

\node[scale=1] (b) at (0,-1) {$\dynkin[scale=1.4] F4$};
\draw[scale=1,->] (a) -- (b) node[midway,left] {$f$};

\node[scale=1] at (4,0.6){
\begin{tikzcd}[column sep=small]
 && \circ \arrow[d,  "\aa"]& && \\ 
 & & \circ \arrow[d,   "\eps"]& & &\\
& & \circ\arrow[d,  "\nu"]&  && \\
 & & \circ\arrow[dl,  swap, "\eps"]\arrow[dr, "\sroot"]& && \\
 &   \circ\arrow[dl,  swap, "\aa"]\arrow[dr, "\sroot"] & &\circ\arrow[dl,   "\eps"] &&\\
 \circ\arrow[dr,  swap, "\sroot"]& &\circ \arrow[dl,  swap,"\aa"]\arrow[dr,  swap,"\nu"]& && \\
& \circ\arrow[dr,  swap, "\nu"]& &\circ\arrow[dl,  swap,"\aa"]\arrow[dr,  swap,"\eps"]&& \\
& & \circ\arrow[dr,  swap, "\eps"]& &\circ\arrow[dl,  swap,"\aa"] &\\
& & &\circ&&\\
& & \circ\arrow[ur, "\eps"]& &\circ\arrow[ul,"\aa"] &\\
& \circ\arrow[ur, "\nu"]& &\circ\arrow[ul,"\aa"]\arrow[ur,"\eps"]&& \\
 \circ\arrow[ur, "\sroot"]& &\circ \arrow[ul,"\aa"]\arrow[ur,"\nu"]& && \\
 &   \circ\arrow[ul,  "\aa"]\arrow[ur, "\sroot"] & &\circ\arrow[ul,   "\eps"] &&\\
 & & \circ\arrow[ul,"\eps"]\arrow[ur, "\sroot"]& && \\
& & \circ\arrow[u,  "\nu"]&  && \\
& & \circ \arrow[u,   "\eps"]& & &\\
 && \circ \arrow[u,  "\aa"]& && 
\end{tikzcd}};

\node[scale=1] at (4,-9){$\hasse{\F_4}\aa$ };

\node[scale=1] at (4,-10){\dynkin[labels={\sroot,\nu,\eps,\aa}, scale=1.4] F4};

\end{tikzpicture}\caption{A surjective map $\T^f:\hasse{\EE_6}\aa\to\hasse{\F_4}\aa$ with labeling $f.$}\label{E6-F4}
\end{figure}



\bibliographystyle{plain}

\

\author{\vbox{\footnotesize\noindent 
	Andr\'es Sambarino\\
	CNRS - Sorbonne Universit\'e et Universit\'e Paris Cit\'e\\ IMJ-PRG\\ 
	4 place Jussieu 75005 Paris France\\
	\texttt{andres.sambarino@imj-prg.fr}
\bigskip}}

\end{document}